\documentclass[11pt,reqno]{amsart}
\usepackage{amssymb}

\newtheorem{theorem}{Theorem}[section]
\newtheorem{lemma}[theorem]{Lemma}
\newtheorem{corollary}[theorem]{Corollary}
\newtheorem{proposition}[theorem]{Proposition}

\newtheorem{stheorem}{Theorem}[subsection]
\newtheorem{slemma}[stheorem]{Lemma}

\newtheorem{sproposition}[stheorem]{Proposition}
\newtheorem{snotation}[stheorem]{Notation}

\theoremstyle{definition}

\newtheorem{itheorem}{Theorem}

\newtheorem{idefinition}[itheorem]{Definition}

\newtheorem{notation}[theorem]{Notation}
\newtheorem{inotation}[itheorem]{Notation}
\newtheorem{conjecture}[theorem]{Conjecture} 
\newtheorem{problem}[theorem]{Problem}

\newtheorem{remark}[theorem]{Remark}

\def\ord{\operatorname{ord}}
\def\C{\mathbb{C}}
\def\R{\mathbb{R}}
\def\Z{\mathbb{Z}}
\def\Q{\mathbb{Q}}
\def\N{\mathbb{N}}

\def\beq {\begin{equation}}
\def\endq {\end{equation}}

\newcommand{\Ocal}{\mathcal{O}}

\begin{document}
\title{Prime power terms in elliptic divisibility sequences.}
\begin{abstract}

We consider the problem of finding explicitly all prime power terms in an elliptic divisibility sequence when descent via isogeny is possible. This question is an analog for elliptic curves to the Mersenne problem.
\end{abstract}
\author{Val\'ery Mah\'e}
\address{I3M, Universit\'e de Montpellier 2; Case Courrier 051, Place Eug\`ene Bataillon; 34095 Montpellier Cedex; France.}
\email{vmahe@math.univ-montp2.fr}
\thanks{This work began at the university of East Anglia (and was funded by an EPSRC grant)  and has been finished at the university of Montpellier 2. I thank Graham Everest for helpful discussions and comments.}

\maketitle

\section{Introduction}

The Mersenne Problem consists in the search for all prime integers of the form $2^{n}-1$ i.e. in the study of integers $n\in\N$ such that the congruence $2^{n}\equiv 1\bmod l$
is satisfied for at most one prime integer $l$. This open problem corresponds to the particular case $G=\mathbb{G}_{m}$ of the following question  : 
\begin{problem}
For any $\Q$-point $P\in G(\Q )$ of an algebraic group $G$ defined over $\Q$ (or over the fraction field of a Dedekind ring), describe the set $\mathcal{I}(P)$ of integers $n$ such that
$$\textrm{Card}\left(\left\{ v\textrm{ place of }\Q~: \textrm{red}_{v}(nP) = \textrm{red}_{v}( 0_{G})\right\}\right)\le 1$$
(where $\textrm{red}_{v}$ denotes the reduction map of $G$ at $v$ and $0_{G}$ denotes the neutral element of $G$). 
\end{problem}
The properties of the set $\mathcal{I}(P)$ depend strongly on the choice for the algebraic group $G$. In fact while the existence of infinitely many Mersenne primes is expected, an analog to the Lenstra-Wagstaff heuristic is considered in \cite{Einsiedler-Everest-Ward}  to suggest the following conjecture is true:

\begin{conjecture}[Primality conjecture] Let $P$ be a point on an elliptic curve $E$ defined over $\Q$ by a Weierstrass equation with integer coefficients. Then the cardinal of the set $\mathcal{I}(P)$ is finite.
\end{conjecture}

In \cite{primeds} Everest, Miller and Stephens prove this primality conjecture for \emph{magnified} points (see below for a definition). This particular case of the primality conjecture is also studied in \cite{Eisentraeger-Everest} as part of a further investigation of a result from Poonen about Hilbert's  tenth problem (see \cite{Poonen}). In \cite{eims} the existence of a uniform upper bound $M$ on $\textrm{Card}(\mathcal{I}(P))$ is proven when $P$ is magnified assuming a conjecture from Lang. 

In this article we give an explicit expression for the bound $M$ as a function in the Szpiro ratio of the underlying elliptic curve $E$. The computation of such a numerical value for $M$ is crucial when considering the problem of sieving for all elements in $\mathcal{I}(P)$. Using the same method we improve also the results proven in \cite{eims} by showing the existence of a uniform bound on $\max\left(\mathcal{I}(P)\right)$ when $P$ is magnified assuming  a conjecture from Lang and a conjecture from Hall and Lang.

\subsection{Background}

A divisibility sequence is a sequence of integers $\left( B_{n}\right)_{n\in\N}$ satisfying the divisibility relation $B_{n}\mid B_{m}$ for every couple $(n,m)\in\N^{2}$ such that $n\mid m$. In \cite{Ward}, Ward study a particular case of divisibility sequences related to the theory of elliptic curves (and division polynomials). 

\begin{inotation}\label{definition-A-P-B-P}
We consider the multiplication-by-$n$ map (denoted $[n]$) on an elliptic curve $E$ defined over $\Q$ by a Weierstrass equation with integral coefficients 
\begin{equation}\label{equation-Weierstrass-intro}
E:y^{2}+a_{1}y+a_{3}xy=x^{3}+a_{2}x^{2}+a_{4}x+a_{6}.
\end{equation}
We denote respectively by $h$, $\widehat{h}$, $h_{\infty}$, $\widehat{h_{\infty}},$ $h_{v}$ and $\widehat{h_{v}}$ the naive height, the canonical height, the archimedean height, the canonical archimedean height, the naive local height at a place $v$ and the canonical local height at $v$. Those heights are defined using the same normalizations as in \cite{advance}.

When the equation (\ref{equation-Weierstrass-intro}) is minimal we denote by $\Delta_{E}$ the discriminant of $E$, by $j(E)$ the $j$-invariant of $E$ and by $h(E)$ the height of $E$ defined as
$$h(E):=\frac{1}{12}\max\left\{ h(j(E)),h(\Delta_{E})\right\} .$$
Let $P$ be a $\Q$-point on $E$ with infinite order. Let $n$ be an integer. We write 
$$[n]P =\left(\frac{A_{nP}}{B_{nP}^{2}},\frac{C_{nP}}{B_{nP}^{3}}\right)$$ 
with $A_{nP}\in\Z$ and $B_{nP}\in\N$ such that $\gcd (A_{nP},B_{nP})=\gcd (C_{nP},B_{nP})=1$. The sequence $(B_{nP})_{n\in\N}$ satisfies the strong divisibility property 
\begin{equation}\label{Strong-Divisibility-Property}
\gcd (B_{nP},B_{mP})=B_{\gcd (n,m)P}.
\end{equation}
\end{inotation}

\begin{idefinition}
We use notation \ref{definition-A-P-B-P}. The sequence $B=(B_{nP})_{n\in\N}$ is called the elliptic divisibility sequence associated to the point $P$. 
\end{idefinition}

This definition is different from the definition given in \cite{Ward}. This slightly different notion of elliptic divisibility sequences appears as a natural tool for the study of the analog of the Mersenne problem for elliptic curves (in fact the set $\mathcal{I}(P)$ is the set of indices of prime power terms in the sequence $B$). In particular we deduce from the  strong divisibility property the existence of natural factorizations of the terms of some elliptic divisibility sequences. As an example when $P=[m]Q$ for some point $Q\in E(\Q )$, Equation (\ref{Strong-Divisibility-Property}) shows that $B_{nQ}$ divides $B_{nP}$ for every integer $n\in\N$. In that case the primality conjecture holds for $P$ if $B_{nQ}$ has a prime factor and $B_{nP}$ has a prime factor coprime to $B_{nQ}$ for all but a finite number of indices $n$. This example can be generalized using the concept of Galois-magnification

\begin{idefinition} We use notation \ref{definition-A-P-B-P}. 
\begin{enumerate}
\item The point $P$ is said Galois-magnified (by Q and $\sigma$) if $P$ can be written as $P=\sigma (Q)$ with 
\begin{itemize}
\item $K_{P,\sigma}$ a finite Galois extension of $\Q$,
\item $F$ an elliptic curve defined over $K_{P,\sigma}$, 
\item $Q$ a point on $F$ defined over $K_{P,\sigma}$,
\item $\sigma :F\longrightarrow E$ an isogeny of degree strictly less than $[K_{P,\sigma}:\Q]$.
\end{itemize}
\item The point $P$ is said magnified if it is Galois-magnified with $K_{P,\sigma}=\Q$.
\item An elliptic divisibility sequence $B$ is said magnified (respectively Galois-magnified) if $B$ is associated to a magnified (respectively Galois-magnified) point. 
\end{enumerate}
\end{idefinition}
The key condition in this definition is the inequality $\deg (\sigma )>[K_{P,\sigma}:\Q]$. It is introduced in \cite{primeds} to study the coprimality of prime factors of $B_{nQ}$ and $B_{nP}$ produced applying strong versions of Siegel's theorem on the finiteness of the set of integral points on an elliptic curve.

\subsection{Statement of the results} 

We consider the problem of computing the set $\mathcal{I}(P)$ when $P$ is magnified over $\Q$. We restrict ourself to the case $K_{P,\sigma}=\Q$ to simplify the statements and the proofs of our results. However more general results for Galois-magnified points could be obtained applying analog methods in the number field case. 

\begin{inotation}\label{notation-P-Pprime-sigma}
Let $E,E'$ be two elliptic curves defined over $\Q$ by standardized minimal Weierstrass equations (Equation (\ref{equation-Weierstrass-intro}) is said standardized when $a_{1},a_{3}\in\{ 0,1\}$ and $a_{2}\in\{-1,0,1\}$). Let $\sigma : E'\longrightarrow E$ be an isogeny defined over $\Q$. Denote by $d$ the degree of $\sigma$.

Let $P'\in E'(\Q )$ be a $\Q$-point on $E'$ with infinite order. Denote by $P$ the image $\sigma (P')$.  Let $(B_{nP'})_{n\in\N}$ (respectively $(B_{nP})_{n\in\N}$) be the elliptic divisibility sequence associated to $P'$ (respectively $P$).
\end{inotation}

Our approach to Mersenne's problem for elliptic curve consists in constructing $\mathcal{I}(P)$ using a set of integer points on an elliptic curve and sets of solutions of some Thue equations. In other words we restate the primality conjecture in the magnified case using classical diophantine equations.

\begin{itheorem}\label{Introduction-Theorem-Thue}
We use notation \ref{notation-P-Pprime-sigma}. 
Let $n$ be an integer such that $B_{n\sigma (P')}$ has at most one prime factor coprime to $B_{P'}.$ Then we have :
\begin{itemize}
\item either $nP'$ is an $S(P')$-integer point (where $S(P')$ denotes the set of prime factors of $B_{P'}$), 
\item or there is an integer $r$ and a divisor $d(n)$ of $\deg (\sigma )^{2}\Delta_{E'}^{r}$ such that 
\begin{enumerate}
\item the integer $d(n)$ varies in a finite set: $|d(n)|\le\deg (\sigma )e^{\left(\frac{3}{2}\deg (\sigma ) h(E')\right)}$
\item and $(A_{nP'},B_{nP'}^{2})$ is a solution of the Thue equation 
$$B_{nP'}^{\deg (\sigma )-1}\psi_{\sigma}(nP') =d(n)$$
\end{enumerate}
where $\psi_{\sigma}$ denotes the division polynomial associated to $\sigma$.
\end{itemize}
\end{itheorem}
An algorithm for the resolution of Thue equations is described in \cite{Tzanakis-De-Weger}. Using Theorem \ref{Introduction-Theorem-Thue} a theorical method for the computation of $\mathcal{I}(P)$ follows. However the statement of Theorem \ref{Introduction-Theorem-Thue} involves a huge number of Thue equations. We deal with this difficulty by adapting results from \cite{Tzanakis-De-Weger}. The main tool is diophantine approximation and especially inequalities of the form
$$\widehat{h_{v_{\infty}}}(P)\le \epsilon\widehat{h}(P) + M$$
where $P$ is a point on an elliptic curve $E$ (defined over a number field $K$), the place $v_{\infty}$ varies among all archimedean places of $K$ and  $\epsilon\in ]0,1[$ and $M>0$ are constants (independent of $P$). In section \ref{section-un} we explain how to deduce from such inequalities an explicit upper bound  $N$ on $\max\left(\mathcal{I}(P)\right)$ that depends only on $h(E),$ $\widehat{h}(P)$, $\epsilon$ and $M$. This leads to a two step method to compute $\mathcal{I}(P)$: 
\begin{enumerate}
\item apply Baker's method to compute explicit values for $\epsilon$ and $M$ (see section \ref{section-quatre});
\item use the bound $N$ to sieve for all prime power terms in $(B_{n\sigma (P' )})_{n\in\N}$. 
\end{enumerate}
We do not insist on step (b) since it can be done using classical sieving algorithm to search for indices $n$ such that either $B_{nP'}$ has no prime factor coprime to $B_{P'}$ or $B_{n\sigma (P')}$ has no prime factor coprime to $B_{nP'}$. We focus instead on Step (a). In section \ref{section-quatre} we prove 
 the following three explicit bounds.
 \begin{itheorem}\label{Bound-compositecase-N1-N3}
We use notation \ref{notation-P-Pprime-sigma}. Let $\mathcal{F}_{E'}$ and $\mathcal{F}_{E}$ (respectively $\Delta_{E'}$ and $\Delta_{E}$) be the conductors (respectively the discriminants) of $E'$ and $E$. Denote by $S_{\sigma} := \max\left\{\frac{\log\left|\Delta_{E'}\right|}{\log\left|\mathcal{F}_{E'}\right|}, \frac{\log\left|\Delta_{E}\right|}{\log\left|\mathcal{F}_{E}\right|}\right\}$ the maximum of the two Szpiro ratios for $E'$~and~$E$. As in \cite{hindrysilverman} we consider the constant $C_{\sigma} : = \max\left\{ 1, (20S_{\sigma})^{8}10^{4S_{\sigma}}\right\}$.
\begin{enumerate}  
\item Let $n$ be an integer such that  at most one prime factor $B_{n\sigma (P')}$ is not a prime factor of $B_{P'}$. Then we have 
$$\textrm{either }n\textrm{ is prime or }n\le\max\left\{ 18C_{\sigma}\left(\log (70 C_{\sigma})\right)^{2}, 490000 C_{\sigma}\right\}$$
\item Denote by $N_{i}$ the $i$-th largest prime index such that  at most one prime factor $B_{N_{i}\sigma (P')}$ is coprime to $B_{P'}$. Then we have 
$$N_{1}\le\max\left\{ 4.2\times 10^{30}C_{\sigma}, 4\times 10^{27}C_{\sigma}^{7/2}\widehat{h}(\sigma (P'))^{5/2}\right\}$$
$$N_{3}\le 77C_{\sigma}$$
\end{enumerate}
\end{itheorem}
If Szpiro's conjecture is true then the upper bound on $N_{3}$ in Theorem \ref{Bound-compositecase-N1-N3} is independent from the choice for $(E,P,\sigma)$. Thus Theorem \ref{Bound-compositecase-N1-N3} gives an explicit version of the main result in  \cite{eims}. In section \ref{section-deux} we deduce from section \ref{section-un} an improvement of the main result in  \cite{eims}: assuming two classical conjectures, we prove the existence of a uniform bound on the index (and not only on the number) of prime power terms in elliptic divisibility sequences. 
\begin{itheorem}\label{theorem-intro-Hall-Lang}
We use notation \ref{notation-P-Pprime-sigma} and  we assume 
\begin{enumerate}
\item the Lang-Silverman conjecture holds (see Conjecture \ref{Lang-Conjecture});
\item the Hall-Lang conjecture holds (see Conjecture \ref{Hall-Lang-Conjecture});
\item that $\deg (\sigma )>4M$ (where $M$ is defined as in Conjecture \ref{Hall-Lang-Conjecture}).
\end{enumerate}
Then there is a constant $N\ge 0$ (independent of $(E, P,\sigma )$) such that $B_{nP}$ has two distinct prime factors coprime to $B_{P'}$ for every index $n>N$.
\end{itheorem}
Note that we need to assume the Hall-Lang conjecture in Theorem \ref{theorem-intro-Hall-Lang} because $n\in\mathcal{I}(P)$ whenever $nP$ is integral. This hypothesis can not be removed in the general case. However if $P$ is in the unbounded component of $E$ or if $P$ is doubly magnified then Theorem \ref{theorem-intro-Hall-Lang} can be stated in an explicit way without assuming the Hall-Lang conjecture (see section \ref{section-cinq} for details).

Section \ref{section-six} focus on the particular case of elliptic curves $E_{A}$ defined over $\Q$ by Weierstrass equations  
$$E_{A}: y^{2}=x(x^{2}-A)$$
(where $A$ is a positive integer with all valuations less than or equal to $4$). Divisibility sequences associated to magnified $\Q$-points in the bounded connected component of $E_{A}(\R )$  tend to have very few prime power terms. In fact, if $P$ is the multiple of a rational point by an odd integer, then $B_{nP}$ has no prime power term with index $n>8$. 

Other families of elliptic divisibility sequences with very few prime power terms can be obtained in the same way using modular curves to parametrize the set of cyclic isogenies between elliptic curves defined over $\Q$. 

\section{Computing the set of indices of prime power terms in magnified elliptic divisibility sequences.}\label{section-trois}

One of the main issue when studying the primality conjecture for magnified points is to compute the image of a given point under a given isogeny in an appropriated way. This can be done using division polynomials to reformulate a formula from V\'elu. For the convenience of the reader we begin by reminding some basic facts on division polynomials.

\subsection{Background on division polynomials.}\label{section-notation}

\begin{snotation}\label{notation-division-polynomial}
We use notation \ref{notation-P-Pprime-sigma}. Let 
$$E:y^{2}+a_{1}y+a_{3}xy=x^{3}+a_{2}x^{2}+a_{4}x+a_{6}$$
$$E':y^{2}+a_{1}'y+a_{3}'xy=x^{3}+a_{2}'x^{2}+a_{4}'x+a_{6}'$$
be the (standardized) minimal Weierstrass equations for $E$ and $E'$. Let $\omega_{E}$ (respectively $\omega_{E'}$) be the (minimal) invariant differential form associated to $E$ (respectively $E'$). Let $d_{\sigma}\in\Q$ be such that $\sigma^{*}\omega_{E}=d_{\sigma}\omega_{E'}.$ We denote 
\begin{itemize}
\item by $\psi_{\sigma}\in\Q (E)$ the division polynomial associated to $\sigma$ i.e. the unique function $\psi_{\sigma}$ on $E$ such that $\psi_{\sigma}^{2}=d_{\sigma}^{2}\displaystyle\prod_{T\in\ker (\sigma )}(x-x(T))$ ;
\item by $\phi_{\sigma}$ the polynomial $\phi_{\sigma}:=\displaystyle\prod_{Q\in E'(\Q), x(\sigma (Q))=0}(x-x(Q))$.
\end{itemize}
\end{snotation}

\begin{slemma}\label{integrality-phi-psi} We use notation \ref{notation-division-polynomial}. Then $d_{\sigma}$ is an element of $\Z$ (equal to $m^{2}$ when $\sigma =[m]$) and for every $P'\in E'(\Q )$ we have 
\begin{equation}\label{defintion-pis-E-sigma}
x(\sigma (P'))=\frac{\phi_{\sigma}(P')}{\psi_{\sigma}(P')^{2}}.
\end{equation}
\end{slemma}
\begin{proof}In \cite{Velu} V\'elu defines an elliptic curve $\mathcal{E}$ using a Weierstrass equation with integral coefficients
$$\mathcal{E}:\widetilde{y}^{2}+\alpha_{1}\widetilde{y}+\alpha_{3}\widetilde{x}\widetilde{y}=\widetilde{x}^{3}+\alpha_{2}\widetilde{x}^{2}+\alpha_{4}\widetilde{x}+\alpha_{6}$$
and an isomorphism $\varphi :E\longrightarrow\mathcal{E}$ such that
$$\widetilde{x}(\varphi(\sigma (P')))=x(P')+\displaystyle\sum_{Q\in \ker(\sigma ),~Q\neq 0}\left(\frac{t_{Q}}{x(P')-x(Q)}+\frac{u_{Q}}{(x(P')-x(Q))^{2}}\right)$$
for every $\Q$-point $P'\notin\ker (\sigma )$ (where $t_{Q}\in\C$ and $u_{Q}\in\C$ are independent from $P'$). Since $E$ is given by a minimal equation and since $\mathcal{E}$ is a model of $E$, we have 
$\widetilde{x}\circ\varphi =s^{2}x+t$ where $s$ and $t$ are two integers. The invariant differential form on $E$ associated to $\mathcal{E}$ is equal to $\varphi^{*}\omega_{\mathcal{E}} = s^{-1}\omega_{E}.$ In \cite{Velu} V\'elu asserts that $(\varphi\circ\sigma )^{*}\omega_{\mathcal{E}}=\omega_{E'}$ i.e. that  $\sigma^{*}\omega_{E}= s\omega_{E'}$. In other words $s=d_{\sigma}$ and it follows that $d_{\sigma}\in\Z$. See \cite[Chapter III, Corollary 5.3]{aec} for the computation of $d_{\sigma}$ when $\sigma =[m]$.

The divisors associated to the two functions $x\circ\sigma$ and $\frac{\phi_{\sigma}}{\psi_{\sigma}^{2}}$ are equals. It follows that $(x\circ\sigma )\frac{\psi_{\sigma}^{2}}{\phi_{\sigma}}$ is an element in $\Q$. In fact using V\'elu's formula to evaluate $\frac{x\circ\sigma}{x}$ at the point at infinity on $E'$ and using the definition of the invariant differential it comes that $x\circ\sigma = \left(\frac{d_{\sigma}}{s}\right)^{2}\frac{\phi_{\sigma}}{\psi_{\sigma}^{2}}=\frac{\phi_{\sigma}}{\psi_{\sigma}^{2}}$.
\end{proof}

\begin{slemma}\label{chain_rule}Let $E,~E',~E''$ be three elliptic curves defined over $\Q$ by Weierstrass equations with integral coefficients. Let $\sigma~:~E~\longrightarrow~E'$ and $\tau~:~E'~\longrightarrow~E''$ 
be two isogenies defined over $\Q$. 
Then the two following equalities hold:
\begin{equation}
\begin{array}{r}
\psi_{\sigma}^{2\deg(\tau )}(\phi_{\tau}\circ\sigma )=\phi_{\tau\circ\sigma};\\
\psi_{\sigma}^{2\deg(\tau )}(\psi_{\tau}\circ\sigma )^{2}=\psi_{\tau\circ\sigma}^{2}.\\
\end{array}
\end{equation}
\end{slemma}
\begin{proof}The formula for $\psi_{\tau\circ\sigma}^{2}$ is obtained by comparing the divisors of $(\psi_{\tau}\circ\sigma )^{2}\psi_{\sigma}^{\deg(\tau )}$ and $\psi_{\tau\circ\sigma}^{2}$; see \cite[appendix 1]{Koblitz} for a more general result. The assertion for $\phi_{\tau\circ\sigma}$ follows since $\frac{\phi_{\tau\circ\sigma}}{\psi_{\tau\circ\sigma}^{2}}=x\circ\tau\circ\sigma =\frac{\phi_{\tau}\circ\sigma}{\psi_{\tau}^{2}\circ\sigma}.$
\end{proof}

\begin{slemma}
We use notation \ref{notation-division-polynomial}. Then the two polynomials $\phi_{\sigma}$ and $\psi_{\sigma}^{2}$ have integral coefficients. 
\end{slemma}
\begin{proof}
Every point on $E$ with $x$-coordinate $0$ is integral. Any preimage of an integral point under an isogeny is integral. In particular the roots of $\phi_{\sigma}$ are integral. This proves that $\phi_{\sigma}$ has integral coefficients. The integrality of the coefficients of $\psi_{\sigma}^{2}$ is a classical generalization of the Nagell-Lutz theorem.
\end{proof}

\begin{snotation}\label{pullback-isogeny}
We keep the hypotheses of Lemma \ref{chain_rule} and we assume that $\deg (\tau )$ and $\deg (\sigma )$ are coprime. We denote by $\widehat{\sigma}$ the dual isogeny of $\sigma$. Then the restriction of $\widehat{\sigma}$ to $\ker (\tau )$ gives a group isomorphism between $\ker (\tau )$ and a $\textrm{Gal}(\overline{\Q}/\Q )$-invariant subgroup of $E [\deg (\tau )]$. This  subgroup $\widehat{\sigma}(\ker (\tau ))$ is the kernel of an isogeny $\tau_{\sigma} : E\longrightarrow E_{\tau_{\sigma}}$ of degree $\deg (\tau )$ where $E_{\tau_{\sigma}}$ denotes an elliptic curve defined over $\Q$ by a standardized minimal equation. Using $\tau_{\sigma}$ a natural factorization of division polynomials can be deduced from Lemma \ref{chain_rule}. 
\end{snotation}   

\begin{slemma}\label{lemma-divisibility-property1}
We use notation \ref{pullback-isogeny} (in particular $\deg (\sigma )$ and $\deg (\tau )$ are coprime). Then $\psi_{\sigma}^{2}\psi_{\tau_{\sigma}}^{2}$ divides $\psi_{\tau\circ\sigma}^{2}$ in $\Z [x]$.
\end{slemma}
\begin{proof}Denote by $\widehat{\sigma}$ the dual isogeny of $\sigma$ and by $\widehat{\tau_{\sigma}}$ the dual isogeny of $\tau_{\sigma}$. Then we have $\ker (\tau\circ\sigma ) = \ker (\sigma ) + \widehat{\sigma}(\ker (\tau )) = \ker (\sigma_{\widehat{\tau_{\sigma}}}\circ\tau_{\sigma}) .$
In particular the definition of division polynomials implies that $\psi_{\tau\circ\sigma} ^{2} = \psi_{\sigma_{\widehat{\tau_{\sigma}}}\circ\tau_{\sigma}} ^{2}.$ Applying Lemma \ref{chain_rule}, we get that  $\psi_{\tau\circ\sigma} ^{2}$ is divisible in $\Q [x]$ by $\psi_{\sigma}^{2}$ and by $\psi_{\tau_{\sigma}}^{2}$. 
The two polynomials $\psi_{\sigma}^{2}$ and $\psi_{\tau_{\sigma}}^{2}$ are coprime because $\ker (\sigma )\cap\ker (\tau_{\sigma})=\{ 0_{E}\} .$ We conclude using two consequences of Cassel's statement for Nagell-Lutz theorem :
\begin{itemize}
\item the quotient $\frac{\psi_{\tau\circ [m]}^{2}}{\psi_{\tau}^{2}\psi_{m}^{2}}$ is an element of $\Z [x]$ because $x(T)$ is an algebraic integer for every $T\in\ker (\sigma ) + \widehat{\sigma}(\ker (\tau ))$ that does not belong to $\ker (\tau )$ or $\widehat{\sigma}(\ker (\tau ))$;
\item the polynomials $\psi_{\sigma}^{2}$ and $\psi_{\tau_{\sigma}}^{2}$ belong also to $\Z [x]$ (see Lemma \ref{integrality-phi-psi}).  
\end{itemize}
\end{proof}

\subsection{Division polynomials and elliptic divisibility sequences.}\label{ssection-Thue-equation}

Elliptic divisibility sequences are closely related to evaluations of division polynomials (see \cite{Ayad,Ward}). For points with good reduction everywhere this link is quite simple.
\begin{stheorem}[Ayad]\label{theoremA-Ayad}Let $v$ be a place of $\Q$. Let $P\in E(\Q)$ be a point on $E$ whose reduction at $v$ is not the reduction at $v$ of $0_{E}$. Then the following assertions are equivalent
\begin{enumerate}
\item the reduction of $P$ at $v$ is a singular point;
\item there is an integer $m$ such that $v(\psi_{m}(P))>0$ and $v(\phi_{m}(P))>0$;
\item for every integer $n$, we have $v(\psi_{n}(P))>0$ and $v(\phi_{n}(P))>0$.
\end{enumerate}
\end{stheorem}
Ayad's theorem does not predict the valuation $v(\psi_{m}(P))$ when $P$ has bad reduction at $v$. 
In \cite{cheon-hahn} the valuations $v(\psi_{m}(P))$ and $v(\phi_{m}(P))$ are studied in terms of the smallest positive integer $N_{P,v}$ such that $N_{P,v}P$ has good reduction at $v$. This integer $N_{P,v}$ can be easily computed using Tate's algorithm (see \cite{advance}). However the computation of an explicit uniform upper bound on the number of prime power terms in magnified elliptic divisibility sequences requires an estimation for $\frac{B_{P}^{2\deg (\sigma )}\psi_{\sigma}(P)^{2}}{B_{\sigma (P)}^{2}}$ that does not depend on $N_{P,v}$. Such an estimation can be obtained from a comparison between naive local heights and their associated canonical local heights. 

\begin{slemma}\label{Parallelogram-law-version-isogeny}
We use notation \ref{notation-division-polynomial}. Then for every $P'\in E'(\Q )$ we have 
\begin{equation}\label{equation-can-height-mP}
\begin{array}{c}
\log|\psi_{\sigma} (P')|=\deg (\sigma )\widehat{h_{\infty}}(P')-\widehat{h_{\infty}}(\sigma (P')) + \frac{\deg (\sigma )\log |\Delta_{E'}|-\log |\Delta_{E}|}{12}.\\
\end{array}
\end{equation}
\end{slemma}
\begin{proof}The proof is based  on \cite[Theorem 6.18]{Everest-Ward-height-entropy} which states that 
$$\widehat{h_{\infty}}(Q) =\displaystyle\lim_{n\rightarrow +\infty}\frac{\log|\psi_{n}(Q))|}{n^{2}} -\frac{1}{12}\log |\Delta_{E}|$$
for any $Q\in E(\Q )$, and the quasiparallelogram law for $\widehat{h_{\infty}}$ which asserts that 
$$\widehat{h_{\infty}}(P+Q)+\widehat{h_{\infty}}(P-Q)=2\widehat{h_{\infty}}(P)+2\widehat{h_{\infty}}(Q)- \log |x(P)-x(Q)|+\frac{1}{6}\log |\Delta_{E}|$$
for every $P,Q\in E(\Q)$ such that $P,Q,P\pm Q\neq 0_{E}.$

When $\sigma =[n]$ Equation (\ref {equation-can-height-mP}) is proven recursively using the quasiparallelogram law for $\widehat{h_{\infty}}$ and the equality $x([n]P)=x(P)-\frac{\psi_{n+1}(P)\psi_{n-1}(P)}{\psi_{n}(P)^{2}}$. This particular case for Equation (\ref {equation-can-height-mP}) and \cite[Theorem 6.18]{Everest-Ward-height-entropy} implies that $\displaystyle\lim_{n\rightarrow +\infty}\frac{\widehat{h_{\infty}}([n]Q)}{n^{2}} = 0 $ for any $Q\in E'(\Q ).$ Hence the quasiparallelogram law for $\widehat{h_{\infty}}$ implies that 
\begin{equation}\label{limite-h_infty-NP-over-Nsquare}
\displaystyle\lim_{n\rightarrow +\infty}\displaystyle\sum_{T\in\ker (\sigma ), T\neq 0}\frac{\log|x([n]P')-x([n]T)|}{n^{2}}=0.
\end{equation}
Applying \cite[Theorem 6.18]{Everest-Ward-height-entropy} to $\sigma (P')$ together with Lemma \ref{chain_rule} we get 
$$\widehat{h_{\infty}}(\sigma (P')) =\displaystyle\lim_{n\rightarrow +\infty}\frac{\log|\psi_{\sigma} ([n]P')\psi_{n}(P')^{\deg (\sigma )}\psi_{\sigma}(P')^{-n^{2}})|}{n^{2}} -\frac{1}{12}\log |\Delta_{E}|.$$ 
From this equality we deduce Equation (\ref{equation-can-height-mP}) in the general case noting that
$$\displaystyle\lim_{n\rightarrow +\infty}\frac{\log|\psi_{\sigma} ([n]P')|}{n^{2}} =\displaystyle\lim_{n\rightarrow +\infty}\displaystyle\sum_{T\in\ker (\sigma ), T\neq 0}\frac{\log|x([n]P')-x([n]T)|}{n^{2}}=0$$
and using Equation (\ref{limite-h_infty-NP-over-Nsquare}) and \cite[Theorem 6.18]{Everest-Ward-height-entropy} (applied to $P'$).
\end{proof}

\begin{sproposition}\label{proposition-link-EDS-division-polynomial}
We use notation \ref{notation-division-polynomial}. 
\begin{enumerate}
\item If $P'$ has good reduction everywhere, then
\begin{equation}\label{equation-comapraison-EDS-div-pol-good-red}
B_{\sigma (P')} = B_{P'}^{\deg (\sigma )}\psi_{\sigma}(P').
\end{equation}
\item In the general case, the quotient $\frac{B_{P}^{\deg (\sigma )}\psi_{\sigma}(P')}{B_{\sigma (P')}}$ satisfies the inequations 
\begin{equation}\label{equation-comapraison-EDS-div-pol}
\log\left| B_{\sigma (P')}\right|\le\log\left|B_{P'}^{\deg (\sigma )}\psi_{\sigma}(P')\right|\le\log\left| B_{\sigma (P')}\right| +\frac{3}{2}\deg (\sigma ) h(E') .
\end{equation}
\end{enumerate}
\end{sproposition}
\begin{proof}
We use the decomposition of the canonical height into a sum of local canonical heights and the equality $\widehat{h}(\sigma (P'))=\deg (\sigma )\widehat{h}(P')$ to reformulate Equation (\ref{equation-can-height-mP}) as
$$\log |\psi_{\sigma}(P')| = \frac{\deg (\sigma )\log |\Delta_{E'}|-\log |\Delta_{E}|}{12} + \displaystyle\sum_{v\textrm{ prime}}\left(\widehat{h}_{v}(\sigma (P'))- \deg (\sigma)\widehat{h}_{v}(P')\right) .$$
Equality (\ref{equation-comapraison-EDS-div-pol-good-red}) follows since 
\begin{equation}\label{hauteur-locale-can-bonne-red}
\widehat{h_{v}}(Q)=\frac{1}{2}\max\{ 0, -v(x(Q))\} +\frac{1}{12}v(\Delta_{\mathcal{E}}) = v(B_{Q})+\frac{1}{12}v(\Delta_{\mathcal{E}})
\end{equation}
for any point $Q\in\mathcal{E}(\Q )$ with good reduction at $v$ (where $\mathcal{E}\in \{ E',E\}$). 

Inequality (\ref{equation-comapraison-EDS-div-pol}) is obtained in the same way as Equality (\ref{equation-comapraison-EDS-div-pol-good-red}) except that we replace Equality (\ref{hauteur-locale-can-bonne-red}) by the following inequality:
$$\frac{1}{24}\min(0,v(j(\mathcal{E})))\le \widehat{h_{v}}(Q)-\frac{1}{2}\max\{ 0,-v(x(Q))\}=\widehat{h_{v}}(Q)-v(B_{Q})\le\frac{1}{12}v(\Delta_{\mathcal{E}})$$
(which holds for any $\Q$-point $Q$ on an elliptic curve $\mathcal{E}$ given by a minimal Weierstrass equation; see \cite[Chapter III, Theorem 4.5]{Lang} for details). 
\end{proof}

Lemma \ref{chain_rule} explain how the division polynomial associated to the composition of two isogenies factorizes in a natural way. The following key lemma gives an analog property for terms in a magnified elliptic divisibility sequence. 

\begin{slemma}\label{easy-factorization} 
We use notation \ref{notation-P-Pprime-sigma}. Recall that $d = \deg (\sigma ).$ Then we have
$$v\left(B_P\right)\le v\left(B_{\sigma(P)}\right).$$
If $E'$ is also minimal, then $v(B_P)>0$ implies
$$v\left(B_{\sigma(P)}\right)\leq v\left(B_P\right)+v(d).$$
\end{slemma}
\begin{proof}
On the assumption that $E'$ is minimal at $v$, it is not hard to show
(see, for example, the exposition
in \cite{streng}) that the isogeny $\sigma$ induces a map of formal groups
$F_\sigma:\hat{E'}\rightarrow \hat{E}$
defined over $\Ocal_v$ with $F_\sigma(0)=0$ (Streng proves this for number
fields, but the proof works for any local field).
It follows immediately that if $v(x(P))<0$, as $F_\sigma(z)\in \Ocal_v[[z]]$
vanishes at $0$,
$$v(B_{\sigma(P)})=v(F_\sigma(z))\geq v(z)=v(B_P).$$
If $E$ is minimal as well, we may apply the same argument to the
dual isogeny $\hat{\sigma}:E\rightarrow E'$,
noting that the composition is the multiplication-by-$d$ map.  The argument above now tell us that $v(B_{\sigma(P)})\leq v(B_{dP})\leq v(B_P)+v(d).$
\end{proof}

\begin{subsection}{The proof of Theorem \ref{Introduction-Theorem-Thue}.}
Assume that $nP'$ is not an $S(P')$-integer point. Then $B_{nP'}$ has a prime factor coprime to $B_{P'}$. Since $B_{nP'}$ divides $B_{n\sigma (P')}$, it follows that every prime factor of $B_{n\sigma (P')}$ divides $B_{nP'}$. Thus Lemma \ref{easy-factorization} implies that $B_{\sigma (nP')}$ divides $\deg (\sigma )B_{nP'}$. Applying Inequality~(\ref{equation-comapraison-EDS-div-pol}) to the point $nP'$ and simplifying by $B_{nP'}$ we get $$\left| B_{nP'}^{\deg (\sigma )-1}\psi_{\sigma}(nP')\right|\le\deg (\sigma )e^{\left(3\deg (\sigma ) h(E')/2\right)} .$$
Moreover from Theorem~\ref{theoremA-Ayad}  we can deduce that $\frac{B_{nP'}^{\deg (\sigma )}\psi_{\sigma}(nP')}{B_{n\sigma (P')}}$ is a divisor of $d_{\sigma}\Delta_{E'}^{r}$ for some integer $r\in\N$. Since $d_{\sigma}$ divides $\deg (\sigma )$ and since $B_{n\sigma (P')}$ divides $\deg (\sigma )B_{nP'}$, we get the divisibility of $\deg (\sigma )^{2}\Delta_{E'}^{r}$ by $B_{nP'}^{\deg (\sigma )-1}\psi_{\sigma}(nP')$. 
\end{subsection}

\section{Prime power terms in elliptic divisibility sequences and Siegel's theorem.}\label{section-un}

In this section we explain how many classical results from the theory of integer points on elliptic curves (especially from transcendence theory) can be used to prove analog results for prime power terms in magnified elliptic divisibility sequences.

We begin with he following lemma which is useful when trying to solve various inequations appearing in the proof of the primality conjecture. The technical introduction of the real number $A$ helps to optimize the size of the bound obtained.

\begin{lemma}\label{solution-bound-nsquare-logn}
Let $a,b$ and $A\ge 1$ be three positive real numbers. Let $n,d\ge 1$ be two positive integers such that 
$$n^{2}\le a(\log (n) +1)^{d} +b.$$
Then we have $n\le \max\left\{ A\left( 2d\log (2d) + 2\log (A)\right)^{d},\frac{a}{A}+\sqrt{b}\right\} .$
\end{lemma}
\begin{proof}
Since $\log (x) \le \frac{x}{2d} + \log (2d) -1$ for every $x\ge 2d$, we have
$$\begin{array}{rcl}
\log\left( A^{1/d}\left( 2d\log (2d) + 2\log (A)\right)\right) &= &\log\left(2d\log (2d) + 2\log (A)\right) + \frac{\log(A)}{d}\\
&\le &2\log (2d) + \frac{2\log (A)}{d} - 1.\\
\end{array}$$
The map $x\longmapsto \log (x) -\frac{x}{A^{1/d}d}$ decreases on $[A^{1/d}d,+\infty [.$ It follows that 
$$\log (n^{1/d}) - \frac{n^{1/d}}{A^{1/d}d}\le \log\left( A^{1/d}\left( 2d\log (2d) + 2\log (A)\right)\right) -2\log (2d) - \frac{2\log (A)}{d}$$
and in particular that $\log (n^{1/d}) \le \frac{n^{1/d}}{A^{1/d}d} -1$ for any integer $d\ge 1$ and any integer $n\ge A\left( 2d\log (2d) + 2\log (A)\right)^{d}.$ From this inequality we deduce from the inequation 
$$n^{2}\le a(\log (n) +1)^{d} +b\le a(d\log (n^{1/d}) +1)^{d} +b$$ 
that either $n^{2}\le \frac{a}{A}n + b$ or $n\le A\left( 2d\log (2d) + 2\log (A)\right)^{d}.$  
\end{proof}

\begin{theorem}\label{theorem-general-bound-PSD-terms}
We use notation \ref{notation-P-Pprime-sigma}.  
Let $M'$, $M$ and $1>\epsilon\ge 0$ be three real numbers such that $d (1-\epsilon )>1$. Let $I$ be the set of indices $n\in\N$ such that
\begin{equation}\label{Strong-Siegel-hypothesis}
\begin{array}{l}
\widehat{h_{\infty}}(nP')\le\epsilon\widehat{h}(nP') +M'\textrm{ and}\\
\widehat{h_{\infty}}(nP)\le\epsilon\widehat{h}(nP) +M.\\
\end{array}
\end{equation}
Then $B_{nP'}$ has a prime factor coprime to $B_{P'}$ for every integer $n\in I$ such that $n\ge 2$ and 
\begin{equation}\label{general-bound-1}
n> \frac{2}{(1-\epsilon )\widehat{h}(P')} + \sqrt{\frac{M' + h(E') + \widehat{h}(P')}{(1-\epsilon )\widehat{h}(P')}}.
\end{equation}
Moreover $B_{n\sigma (P')}$ has a prime factor coprime to $B_{\sigma (P')}B_{nP'}$ for any $n\in I$ such that $n\ge 2$ and 
\begin{equation}\label{general-bound-2}
n> \frac{2}{(d-d\epsilon -1)\widehat{h}(P')} + \sqrt{\frac{M+h(E)+d\widehat{h}(P') + \log (d)}{(d-d\epsilon -1)\widehat{h}(P')}}.
\end{equation}
\end{theorem}
\begin{proof}
The key ingredient is Inequalities (\ref{Strong-Siegel-hypothesis}) which play a role analog to Roth's theorem in the classical proof of Siegel's theorem.

Let $n\in I$ be an integer such that every prime factor of $B_{nP'}$ divides $B_{P'}$. The quadraticity of $\widehat{h}$ and the decomposition of $\widehat{h}$ into local canonical heights gives
$$n^{2}\widehat{h}(P')=\widehat{h}(nP')=\widehat{h_{\infty}}(nP')+\displaystyle\sum_{v\left( B_{nP'}\Delta_{E'}\right) >0}\widehat{h_{v}}(nP').$$
This equality and Inequality (\ref{Strong-Siegel-hypothesis}) imply that 
\begin{equation}\label{theorem-general-bound-PSD-terms-1st-part-ineq0}
(1-\epsilon )n^{2}\widehat{h}(P')\le M' + \displaystyle\sum_{v\left( B_{nP'}\Delta_{E'}\right) >0}\widehat{h_{v}}(nP').
\end{equation}
Using  \cite[Chapter III, Theorem 4.5]{Lang} Inequality (\ref{theorem-general-bound-PSD-terms-1st-part-ineq0}) becomes
\begin{equation}\label{theorem-general-bound-PSD-terms-1st-part-ineq2}
(1-\epsilon )n^{2}\widehat{h}(P')\le M' + h(E') + \displaystyle\sum_{v\left( B_{nP'}\right) >0}\widehat{h_{v}}(nP')
\end{equation}
Let $v$ be a place such that $v(B_{nP'})>0$. Then our hypothesis on the prime factors of $B_{nP'}$ asserts that $v(B_{P'})>0$. In particular $P'$ and $nP'$ have good reduction at $v$. It follows that $\widehat{h_{v}}(nP')=h_{v}(nP')+\frac{v(\Delta_{E'})}{12}$ and $\widehat{h_{v}}(P')=h_{v}(P')+\frac{v(\Delta_{E'})}{12}$. Since $v(B_{P'})>0$ Lemma~\ref{easy-factorization} implies that
$$\widehat{h_{v}}(nP')\le\widehat{h_{v}}(P') + 2h_{v}(n).$$
We deduce from this inequality and Inequality (\ref{theorem-general-bound-PSD-terms-1st-part-ineq2}) that
$$(1-\epsilon )n^{2}\widehat{h}(P')\le M' + h(E') + \displaystyle\sum_{v\left( B_{nP'}\right) >0}\left(\widehat{h_{v}}(P') + 2h_{v}(n)\right).$$
Using the inequality $\widehat{h_{\infty}}(P')\ge 0$ (see \cite[Chapter III, Theorem 4.5]{Lang}) we get 
\begin{equation}\label{theorem-general-bound-PSD-terms-1st-part-ineq3}
(1-\epsilon )n^{2}\widehat{h}(P')\le M' + h(E') + \widehat{h}(P') + 2\log (n).
\end{equation}
if $n\ge 2\log (2)$, applying Lemma \ref{solution-bound-nsquare-logn} with $A=1$, Inequality (\ref{theorem-general-bound-PSD-terms-1st-part-ineq3}) becomes
$$n\le \frac{2}{(1-\epsilon )\widehat{h}(P')} + \sqrt{\frac{M' + h(E') + \widehat{h}(P')}{(1-\epsilon )\widehat{h}(P')}}.$$

Let $n\in I$ be an integer such that every prime factor of $B_{nP}$ divides $B_{nP'}B_{\sigma (P')}$. The computations above are valid with $P'$ replaced by $P$ and $E'$ by $E$. We get an analog to Inequality (\ref{theorem-general-bound-PSD-terms-1st-part-ineq2}): 
\begin{equation}\label{theorem-general-bound-PSD-terms-2nd-part-ineq2}
(1-\epsilon )n^{2}\widehat{h}(\sigma (P'))\le M + h(E) + \displaystyle\sum_{v\left( B_{nP'}B_{\sigma (P')}\right) >0}\widehat{h_{v}}(n\sigma (P'))
\end{equation}
Lemma \ref{easy-factorization} implies that 
$$\displaystyle\sum_{v\left( B_{nP'}\right) >0}\widehat{h_{v}}(n\sigma (P'))\le\left(\displaystyle\sum_{v\left( B_{nP'}\right) >0}\widehat{h_{v}}(nP')\right) + \log (d)\le \widehat{h}(nP') +\log (d)$$
and $\displaystyle\sum_{v\left( B_{\sigma (P')}\right) >0}\widehat{h_{v}}(n\sigma (P'))\le\widehat{h}(\sigma (P')) +2\log (n).$Thus Inequality (\ref{theorem-general-bound-PSD-terms-2nd-part-ineq2}) gives
\begin{equation}\label{theorem-general-bound-PSD-terms-2nd-part-ineq3}
(d-d\epsilon -1)n^{2}\widehat{h}(P')\le M + h(E) + \widehat{h}(\sigma (P')) + 2\log (n) + \log (d).
\end{equation}
If $n\ge 2\log (2)$ we deduce from Lemma \ref{solution-bound-nsquare-logn} applied with $A=1$ that
$$n\le \frac{2}{(d-d\epsilon -1)\widehat{h}(P')} + \sqrt{\frac{M+h(E)+\widehat{h}(\sigma (P')) + \log (d)}{(d-d\epsilon -1)\widehat{h}(P')}}. $$
\end{proof}

To deduce from Theorem \ref{theorem-general-bound-PSD-terms} a uniform bound on the indices of prime power terms in magnified elliptic divisibility sequences one need to compare the naive heights $h(E')$ and $h(E)$ of two isogenous elliptic curves $E'$ and $E$. Such a comparison follows from the good behaviour of the Faltings height under isogeny. 
\begin{proposition}\label{lemma-pellarin}
We use notation \ref{notation-P-Pprime-sigma}. Then we have 
$$h(E')\le \alpha h(E)+ h(\deg (\sigma )) + 15.8.$$
with $\alpha = 5$ if $h(j(E))>4$ and $\alpha =16$ if $h(j(E))\le 4.$
\end{proposition}
\begin{proof}The proof is based on the good behaviour of the Faltings height $h_{F}$ under isogeny: if $\sigma :E'\longrightarrow E$ is an isogeny between elliptic curves, then the Faltings heights $h_{F}(E)$ of $E$ and $h_{F}(E')$ of $E'$ satisfy the inequality:
$$\left| h_{F}(E)-h_{F}(E')\right|\le\frac{1}{2}\log (\deg (\sigma )).$$
When $E$ is a semi-stable elliptic curve, an explicit bound on  the difference between the Faltings height $h_{F}(E)$ of $E$ and the height $h(j(E))$ can be found in \cite{Pellarin}. In the general case, the proof of \cite[Lemma 5.2]{Pellarin}) gives
$$12h_{F}(E)\le\log\max\{ |j(E)\Delta_{E}|, |\Delta_{E}|\} + 6\log (1+h(j(E)) +47.15$$
$$\log\max\{ |j(E)\Delta_{E}|, |\Delta_{E}|\}\le 94.3 +24\max\{ 1,h_{F}(E)\} .$$
The term $\log\max\{ |j(E)\Delta_{E}|, |\Delta_{E}|\}$ can be expressed in terms of $h(E)$ using the two inequalities:
$$\begin{array}{lll}
12h(E)&=\max \{ h(\Delta_{E}),h(j(E))\}&\\
&\le\log\max\{ |j(E)\Delta_{E}|, |\Delta_{E}|\}&\le 24h(E).\\
\end{array}$$
It follows that
$$\begin{array}{rcl}
12h(E')
&\le &24\max\{ 1,h_{F}(E')\} +94.3 \\
&\le &\max\{ 24, 24h_{F}(E) + 12\log (\deg (\sigma ))\} + 94.3 \\
&\le &48h(E)+ 12\log (1+h(j(E)) + 12\log (\deg (\sigma )) + 188.6\\
\end{array}$$
We conclude by noticing that $\log (1+h(j(E)))\le\frac{h(j(E))}{12}\le h(E)$ whenever the inequality $h(j(E)) >48$ holds. 
\end{proof}

\section{A consequence of the Hall-Lang conjecture.}\label{section-deux}

In \cite{eims}, Everest, Ingram, Stevens and the author prove the existence of a uniform bound on the number of prime power terms in magnified elliptic divisibility sequences assuming the following conjecture.
\begin{conjecture}[Lang]\label{Lang-Conjecture}
There is an (absolute) constant $C>0$ such that for every $\Q$-point $P$ on an elliptic curve $E$ defined over $\Q$ by a minimal equation the following inequality holds
$$h(E)\le C\widehat{h}(P).$$
\end{conjecture}
Using Theorem \ref{theorem-general-bound-PSD-terms} many bounds on integer points of elliptic curve can be generalized to the case of prime power terms in magnified elliptic divisibility sequences. In this section we show an improvement of the main result proven in \cite{eims}: the existence of a uniform bound on the indices of prime power terms in magnified elliptic divisibility sequences, assuming the Lang conjecture and the following conjecture.
\begin{conjecture}[Hall, Lang]\label{Hall-Lang-Conjecture}There are two constant $K, M>0$ such that for every quadruplet of integers $(A,B,x,y)$ with $y^{2}=x^{3}+Ax +B$ the following inequality holds 
\begin{equation}
\max\{ |x|,|y|\}\le K\max\{ |A|,|B|\}^{M}.
\end{equation}
\end{conjecture}
Given a point $P$ on an elliptic curve, the multiple $nP$ is an integer point if and only if the $n$-th term in the elliptic divisibility sequences associated to $P$ is a unit (i.e. has no prime factor). This explains why we need the Hall-Lang conjecture to prove the existence of a uniform bound on the set of indices $n$ such that $B_{nP}$ has at most one prime factor. \\

\begin{proposition}\label{All-lang-Hall-Lang}
We use notation \ref{notation-P-Pprime-sigma} and  we assume 
\begin{enumerate}
\item the Hall-Lang conjecture;
\item the Lang conjecture;
\item that $E$ and $E'$ are given by minimal short Weierstrass equations;
\item that $\deg (\sigma )>4M$ (where $M$ is defined as in the statement of the Hall-Lang conjecture).
\end{enumerate}
Then there is a constant $N\ge 0$ (independent of the choice for $(E,P,\sigma )$) such that $B_{nP}$ has two distinct prime factors coprime to $B_{P'}$ for each $n>N$.
\end{proposition}
\begin{proof}
Let $A,B,A',B'$ be four integers such that $E$ and $E'$ are given by the two following equations
$$\begin{array}{l}
E:y^{2}=x^{3}+Ax+B\\
E':y^{2}=x^{3}+A'x+B'\\
\end{array}$$
Let $n\ge 3$ be an integer. Since $(A_{nP'},C_{nP'})$ is an integer point on the elliptic curve given by the equation $y^{2}=x^{3}+A'B_{nP'}^{4} + B'B_{nP'}^{6}$, the Hall-Lang conjecture gives
$$\frac{1}{2}\log |A_{nP'}|\le 3M\log (B_{nP'}) + \frac{M}{2}\log\max\{ |A'|,|B'|\}  + \frac{1}{2}\log (K)$$
which can be rephrased as 
\begin{equation}\label{Hall-Lang-devenu-Roth}
h(nP')\le 3M (h(nP')-h_{\infty}(nP')) + 6Mh(E') +\frac{1}{2}\log (K)
\end{equation}
Using the two inequalities $\widehat{h_{\infty}}(Q)\le h_{\infty}(Q) + \frac{h(j(E'))}{12} +1.07$ and \linebreak$h(Q) \le\widehat{h}(Q) + \frac{h(j(E'))}{8} +1.205$ (proven in \cite{Silverman-diff-height-can-height}) Inequality (\ref{Hall-Lang-devenu-Roth}) becomes 
$$\begin{array}{rcl}
\widehat{h_{\infty}}(nP')&\le &\left( 1-\frac{1}{3M}\right)\widehat{h}(nP')+ 4 h(E') + \frac{\log (K)}{6M}+ 1.88\\
\end{array}$$
In the same way we prove that 
$$\widehat{h_{\infty}}(nP)\le \left( 1-\frac{1}{3M}\right)\widehat{h}(nP)+4h(E) + \frac{\log (K)}{6M}+ 1.88.$$
We apply Theorem \ref{theorem-general-bound-PSD-terms} noting that if $\deg (\sigma )>4M$ then $\frac{\log (\deg (\sigma ))}{\deg (\sigma )-3M}\le 4$ and $\frac{h(E')}{\widehat{h}(P')}\le C$ and $\frac{h(E)}{\widehat{h}(\sigma (P'))}\le C$ and $\frac{1}{\widehat{h}(P')}\le\frac{C}{h(E')}\le\frac{12C}{\log (2)}.$
In particular we get the existence of a function $N:\R^{3}\longrightarrow\R^{+}$ (independent of the choice for $(E,P,\sigma)$) such that if at most one prime factor of $B_{n\sigma (P')}$ divides $B_{P'}$ then $n\le N\left( M,\log (K), C\right) .$
\end{proof}

\section{Elliptic divisibility sequences associated to points in the bounded connected component of an elliptic curve.}\label{section-cinq}

We study Theorem \ref{theorem-intro-Hall-Lang} for two examples of magnified elliptic divisibility sequences:
\begin{itemize}
\item first we study the case when $P$ is in the unbounded component of $E$;
\item them we consider the case when $P$ is doubly magnified.
\end{itemize}
In those two particular cases we prove that Theorem \ref{theorem-intro-Hall-Lang} holds even if  the Hall-Lang conjecture is false. The results obtained in this section will be used in the proof of theorem \ref{Bound-compositecase-N1-N3}.

\begin{notation}\label{notation-minimal-equation-short-Weierstrass-equation-1}
Let $E$ be an elliptic curve defined over $\Q$ by a minimal Weierstrass equation. This minimal equation might not be a short Weierstrass equation, but the elliptic curve $E$ is isomorphic to an  elliptic curve $\mathcal{E}$ given by a short Weierstrass equation
$$\mathcal{E}: \widetilde{y}^{2}=\widetilde{x}^{3}+a\widetilde{x}+b$$
where $a$ and $b$ are two integers such that $\Delta_{\mathcal{E}}=6^{12}\Delta_{E}.$ The heights of $\mathcal{E}$ and $E$ are related by two inequalities $h(E)\le h(\mathcal{E})\le h(E)+\log (6)$. Since 
$$\begin{array}{rcl}
h(4a^3)=h\left(\frac{j(\mathcal{E})\Delta_{\mathcal{E}}}{16\times 12^{3}}\right)&= &h\left( 4\times 3^{9}\times j(\mathcal{E})\times \Delta_{E}\right)\\
&\le &h(j(\mathcal{E}))+h(\Delta_{E}) + 2\log (2) + 9\log (3)\\
\end{array}$$

$$\begin{array}{rcl}
\textrm{and }h(27b^{2})&=&h\left(\frac{\Delta_{\mathcal{E}}}{16}-4a^{3}\right)\\
&\le &\max\{ h(\Delta_{E}) + 8\log (2) + 12\log (3),h(4a^{3})\} + \log (2)\\
\end{array}$$
the following inequality holds
\begin{equation}\label{bound-David-height-naive-height}
\begin{array}{c}
\max\left\{ 1,h\left( 1,-\frac{a}{4},-\frac{b}{16}\right) ,h(j(\mathcal{E}))\right\}\le 12h(E) + 5\log (6).\\
\end{array}
\end{equation}

The left handside in Inequality (\ref{bound-David-height-naive-height}) appears in David's lower bound on linear forms in elliptic logarithm \cite[Th\'eor\`eme 2.1]{David}, a result used in section~\ref{section-quatre}. 
\end{notation}

\begin{proposition}\label{weakerthan-canonical-height}Let $E$ be an elliptic curve defined over $\Q$ by a minimal Weierstrass equation. We assume that $E(\R )$ has two connected components. Then for every rational point  $Q$ in the bounded connected component of $E (\R )$ the following inequality holds:
\begin{equation}
\widehat{h_{\infty}}(Q)\le 3h(E)+ \log (6) +1.07. 
\end{equation}
\end{proposition}
\begin{proof}
We use notation \ref{notation-minimal-equation-short-Weierstrass-equation-1}. Denote by $\alpha_{1}, \alpha_{2},\alpha_{3}$ the three roots of the polynomial $\widetilde{x}^{3}+a\widetilde{x}+b.$ Following the Cardan Formula there are two complex numbers $u_{i},v_{i}$ such that $\alpha_{i}=u_{i}+v_{i}$ and
$$\Delta_{\mathcal{E}} =-16\times 27\times (b + 2u_{i}^{3})^{2}=-16\times 27\times (b + 2v_{i}^{3})^{2}.$$
Since $-16\times 27b^{2} = \frac{(j(\mathcal{E})+1728 )\Delta_{\mathcal{E}}}{1728}$ and $\Delta_{\mathcal{E}}=6^{12}\Delta_{E}$ and $j(\mathcal{E})=j(E)$ we have
$$\begin{array}{rcl}
2|u_{i}|^{3}\le |b| +|b+2u_{i}^{3}|&\le &e^{6h(E)+6\log (6)}\left(\frac{1}{2^{4}\times 3^{3}}+\frac{e^{12h(E)}}{2^{10}\times 3^{6}}\right)^{1/2}\\
&&+\frac{e^{6h(E)+6\log (6)}}{12\sqrt{3}}\\
&\le &\frac{e^{6h(E)+6\log (6)}}{12\sqrt{3}} + \frac{e^{12h(E)+6\log (6)}}{864} +\frac{e^{6h(E)+6\log (6)}}{12\sqrt{3}}\\
&\le & \frac{e^{12h(E)+6\log (6)}}{4\sqrt{3}}.\\
\end{array}$$
In the same way, we prove that $|v_{i}|\le
\frac{e^{4h(E)+2\log (6)}}{2\times 3^{1/6}}$. An upper bound for $|\alpha_{i}|$ follows: $|\alpha_{i}|\le
\frac{e^{4h(E)+2\log (6)}}{3^{1/6}}.$
Since $|x(Q)|\le\displaystyle\max_{i=1}^{3}(|\alpha_{i}|)$ for every point
$Q$ in the bounded real connected component of $\mathcal{E}$ we get 
$$h_{\infty}(Q)\le 2 h(E) +\log (6)$$
We conclude by applying \cite[Theorem 5.5]{Silverman-diff-height-can-height} which asserts that 
$$\widehat{h_{\infty}}(Q)\le h_{\infty}(Q) + \frac{1}{12}h(j(\mathcal{E})) +1.07$$
for every point $Q\in\mathcal{E}(\Q )$.
\end{proof}
\begin{remark}
We keep the notation of the proof. While the archimedean height $h_{\infty}$ might not be the same for $E$ and for $\mathcal{E}$, the canonical archimedean $\widehat{h_{\infty}}$ does not depend on the choice of a model for the elliptic curve $E$.
\end{remark}

Now we consider the primality conjecture for an elliptic divisibility sequence associated to a point $P$ that is magnified by an isogeny $\sigma$ and a point $P'$ which assumed to be magnified. This case will be used to study  the primality conjecture for elliptic divisibility sequences associated to points belonging to the bounded real connected component of an elliptic curve.
\begin{proposition}\label{Siegel-fort-cas-magnifie}
We use notation \ref{notation-P-Pprime-sigma}. Let $\tau :E''\longrightarrow E$ be either an isogeny defined over $\Q$ (with $E''$ an elliptic curve defined over $\Q$ by a standardized minimal equation) or the identity map. If every prime factor of $B_{\sigma (\tau (P'))}$ divides $B_{\tau (P')}$, then 
\begin{equation}\label{equation-Siegel-fort-cas-doublement-maginifie}
\widehat{h_{\infty}}(P')\le 7h(E') + 8 + \log (\deg (\sigma\circ\tau )). 
\end{equation}
\end{proposition}
\begin{proof}
Assume every prime factor of $B_{\sigma (\tau (P'))}$ divides $B_{\tau (P')}$. Let $T_{0}~\notin~\ker (\tau)$ be a a $\sigma\circ\tau$-torsion point such that  $|x(P')-x(T_{0})|\le |x(P')- x(T)|$ for every $\sigma\circ\tau$-torsion point $T\notin\ker (\tau )$. Since the leading coefficient $d_{\sigma\circ\tau}$ of $\psi_{\sigma\circ\tau}$ is an integer divisible by the leading coefficient $d_{\tau}$ of $\psi_{\tau}$we have 
$$\left| x(P') -x(T_{0})\right|^{\deg (\sigma\circ\tau )-\deg (\tau )}\le\frac{d_{\sigma\circ\tau}^{2}}{d_{\tau}^{2}}\displaystyle\prod_{T\in\ker (\sigma\circ\tau ), T\notin\ker (\tau )}\left| x(P')- x(T)\right| = \frac{\psi_{\sigma\circ\tau}^{2}(P')}{\psi_{\tau}^{2}(P')}$$
From this inequality and Proposition \ref{proposition-link-EDS-division-polynomial} we deduce that
$$\begin{array}{rcl}
\left| x(P') - x(T_{0})\right|^{\deg (\sigma\circ\tau )-\deg (\tau )}&\le &\frac{\left| B_{\sigma\circ\tau (P')}\right|^{2}}{\left|\psi_{\tau} (P')\right|^{2} \left| B_{P'}\right|^{2\deg (\sigma\circ\tau )}}e^{3\deg (\sigma\circ\tau ) h(E')}\\
&\le &\frac{\left| B_{\sigma\circ\tau (P')}\right|^{2}}{\left| B_{\tau (P')}\right|^{2}}e^{3\deg (\sigma\circ\tau ) h(E')}.\\
\end{array}$$
Applying Lemma \ref{easy-factorization} we get
$$\begin{array}{rcl}
\left| x(P') - x(T_{0})\right|^{(\deg (\sigma )-1)\deg (\tau )}&\le &\deg (\sigma )^{2}e^{3\deg (\sigma\circ\tau ) h(E')}\\
&\le &e^{2\deg (\sigma )-2}e^{3\deg (\sigma\circ\tau ) h(E')}\\
\end{array}$$
and in particular $|x(P')-x(T_{0})|\le e^{2+6 h(E')}.$ The triangular inequality gives
\begin{equation}\label{Roth-final-cas-doublement-magnifie}
|x(P')|\le 2\max\left\{| x(T_{0})|,e^{2+6 h(E')}\right\}
\end{equation}
Let $\mathcal{E}'$ be the model for $E'$ deduced from the change of variable $(\widetilde{x},\widetilde{y})= (36x + 3a_{1}^{2} + 12 a_{2}, 216y +108 a_{1}x + 108 a_{3})$. The curve $\mathcal{E}'$ is also the model for $E'$ considered in notation \ref{notation-minimal-equation-short-Weierstrass-equation-1}. Inequation (\ref{bound-David-height-naive-height}) and \cite[Lemme 10.1]{David} give
$$\left| 36x(T_{0})\right| -15 \le\left| 36x(T_{0}) + 3a_{1}^{2} + 12 a_{2}\right|\le 480\deg (\sigma\circ\tau )^{2}e^{12h(E') + 5\log (6)}$$
It follows from Inequality (\ref{Roth-final-cas-doublement-magnifie}) and \cite[Theorem 5.5]{Silverman-diff-height-can-height} that 
$$\widehat{h_{\infty}}(P')\le h_{\infty}(P') + h(E')+\frac{1}{2}\log (2) + 1.07\le 7h(E') + 8 + \log (\deg (\sigma\circ\tau ))$$
(note that \cite[Theorem 5.5]{Silverman-diff-height-can-height} is applied to a standardized equation and that $h_{\infty}(P')=\frac{1}{2}\log\max\{ 1,|x(P')|\}$).
\end{proof}

\begin{corollary}\label{application-cas-doublement-magnifie}
Let $E_{0}, E_{1}, E_{2}, E_{3}$ be four elliptic curves defined over $\Q$ by standardized minimal equations. For each $i\in\{ 1,2,3\}$ let $\tau_{i}: E_{i-1}\longrightarrow E_{i}$ be an isogeny defined over $\Q$. Let $P'\in E_{0}(\Q )$ be a point with infinite order such that $B_{(\tau_{3}\circ\tau_{2}\circ\tau_{1})(P')}$ has two distinct prime factors coprime to $B_{P'}$. Then for each index $i$ we have
$$\begin{array}{cl}
\textrm{either }&\sqrt{\deg (\tau_{i})} \le\frac{2}{\sqrt{\widehat{h}(P')}}\log\left(\frac{2}{\sqrt{\widehat{h}(P')}}\right)\\
\textrm{ or }&\sqrt{\deg (\tau_{i})}\le\frac{144}{\sqrt{\widehat{h}(P')}}+ 2\sqrt{1+ \frac{128h(E_{0}) + 135}{\widehat{h}(P')}} .\\
\end{array}$$
\end{corollary}
\begin{proof}We denote by $d_{i}$ the degree $d_{i} :=\deg (\tau_{i})$ of $\tau_{i}$. Replacing $\tau_{i}$ with $\left( \tau_{i+1}\right)_{\tau_{i}}$ if needed (see notation \ref{pullback-isogeny} for details), we can assume without loss of generality that $d_{1}\ge\max\{d_{2}, d_{3}\} .$ 

Assume for now that $l$ divides $B_{\tau_{1}(P')}$. Following lemma \ref{easy-factorization}, the prime $l$ divides  $B_{\left(\tau_{2}\circ\tau_{1}\right)(P')}$. Thus each prime factor of $B_{\left(\tau_{3}\circ\tau_{2}\circ\tau_{1}\right)(P')}$ divides $B_{\left(\tau_{2}\circ\tau_{1}\right)(P')}$. Since $\log (d_{3})\le\log (d_{1})$, Proposition \ref{Siegel-fort-cas-magnifie} gives 
$$\widehat{h_{\infty}}\left(\left(\tau_{2}\circ\tau_{1}\right)(P')\right)\le 7h(E_{2}) + 8 + \log (d_{1}).$$
Each prime factor of $B_{\left(\tau_{2}\circ\tau_{1}\right)(P')}$ divides $B_{\tau_{1}(P')}$. In particular the following analog to Inequation (\ref{theorem-general-bound-PSD-terms-2nd-part-ineq3}) holds: 
$$\frac{d_{1}d_{2}\widehat{h}(P')}{4}\le\left( d_{2}-1\right)(d_{1}-1)\widehat{h}(P')\le 8h(E_{2}) + 8 + \log (d_{1}) + \log (d_{2}) + \widehat{h}(P').$$
Following Proposition \ref{lemma-pellarin} this inequality implies that 
$$d_{1}d_{2}\widehat{h}(P')\le 4\times \left(128h(E_{0}) + 135 + 9\log (d_{1}d_{2}) + \widehat{h}(P')\right) .$$
Applying Lemma \ref{solution-bound-nsquare-logn} with $n=\sqrt{d_{1}d_{2}}$ and $A=\frac{1}{\sqrt{\widehat{h}(P')}}$ we get that either $\sqrt{d_{1}d_{2}}\le \frac{2}{\sqrt{\widehat{h}(P')}}\log\left(\frac{2}{\sqrt{\widehat{h}(P')}}\right)$ or $\sqrt{d_{1}d_{2}}\le \frac{144}{\sqrt{\widehat{h}(P')}} + 2\sqrt{1+ \frac{128h(E_{0})+135}{\widehat{h}(P')}}.$

Assume now that $l$ does not divide $B_{\tau_{1}(P')}$. If $l$ does not divide $B_{\tau_{2}\circ\tau_{1}(P')}$, then every prime factor of $B_{\left(\tau_{2}\circ\tau_{1}\right)(P')}$ divides $B_{\tau_{1}(P')}$. In that case, since $\log (d_{2})\le\log (d_{1})$, Proposition \ref{Siegel-fort-cas-magnifie} gives 
$$\widehat{h_{\infty}}\left(\tau_{1}(P')\right)\le 7h(E_{1}) + 8 + \log (d_{1}).$$
if $l$ divides $B_{\tau_{2}\circ\tau_{1}(P')}$ then every prime factor of $B_{\tau_{3}\circ\tau_{2}\circ\tau_{1}(P')}$ divides $B_{\tau_{2}\circ\tau_{1}(P')}$. In that case, since $\log (d_{2}d_{3})\le 2\log (d_{1})$, 
Proposition \ref{Siegel-fort-cas-magnifie} (applied with $\sigma :=\tau_{3}$ and $\tau :=\tau_{2}$) gives
$$\widehat{h_{\infty}}\left(\tau_{1}(P')\right)\le 7h(E_{1}) + 8 + 2\log (d_{1}).$$
In both cases, $l$ being coprime to $B_{\tau_{1}(P')}$, each prime factor of $B_{\tau_{1}(P')}$ divides $B_{P'}$. In particular the following analog to inequality (\ref{theorem-general-bound-PSD-terms-1st-part-ineq3}) follows using Proposition \ref{lemma-pellarin}:
$$d_{1}\widehat{h}(P')\le 128h(E_{0}) + 135 + 10\log (d_{1}) + \widehat{h}(P')$$
Applying Lemma \ref{solution-bound-nsquare-logn} with $n=\sqrt{d_{1}}$ and $A=\frac{1}{\sqrt{\widehat{h}(P')}}$ we get that either $\sqrt{d_{1}}~\le~\frac{2}{\sqrt{\widehat{h}(P')}}\log\left(\frac{2}{\sqrt{\widehat{h}(P')}}\right)$ or $\sqrt{d_{1}}\le\frac{20}{\sqrt{\widehat{h}(P')}} + \sqrt{1+ \frac{128h(E_{0})+135}{\widehat{h}(P')}}.$
\end{proof}

\begin{corollary}\label{cas-composante-bornee}
We use notation \ref{notation-P-Pprime-sigma}. We assume that $E(\R )$ has two connected component and that $\deg (\sigma )$ is odd. We assume that $P=\sigma (P')$ belongs to the bounded connected component of $E (\R)$. Then $B_{nP}$ has two distinct prime factors coprime to $B_{P'}$ for every integer $n$ such that 
$$\begin{array}{cl}
\textrm{either }&n >\frac{4}{\sqrt{\widehat{h}(P')}}\log\left(\frac{2}{\sqrt{\widehat{h}(P')}}\right)\\
\textrm{ or }&n>\frac{288}{\sqrt{\widehat{h}(P')}}+ 4\sqrt{1+ \frac{128h(E') + 135}{\widehat{h}(P')}} .\\
\end{array}$$
\end{corollary}	
\begin{proof}
When $n$ is even Corollary \ref{cas-composante-bornee} follows from Corollary \ref{application-cas-doublement-magnifie} applied with $\tau_{1}= n/2$ and $\tau_{2} = 2. $ We assume now that $n$ is odd.

Since $\sigma$ is an isogeny with odd degree and $E(\R )$ has two connected components, $E'(\R )$ has also two connected components. Moreover, $\sigma (P')$ being on the bounded connected component of $E(\R )$, the point $P'$ is on the bounded connected component of $E'(\R ).$

The index $n$ being odd,  the points $nP'$ and $nP=n\sigma (P')$ are respectively in the bounded connected components of $E'(\R )$ and $E(\R )$. Since $\deg (\sigma )\ge 3$, applying Proposition \ref{weakerthan-canonical-height}, we get two analogs to inequalities (\ref{theorem-general-bound-PSD-terms-1st-part-ineq3}) and (\ref{theorem-general-bound-PSD-terms-2nd-part-ineq3}):
$$n^{2}\widehat{h}(P')\le 4h(E') + \widehat{h}(P') + 2\log (n) +\log (6) + 1.07$$
if every prime factor of $B_{nP'}$ divides $B_{P'}$ and (using Proposition \ref{lemma-pellarin})
$$\begin{array}{rcl}
n^{2}(\deg (\sigma ) -1)\widehat{h}(P')&\le &4h(E) + \widehat{h}(\sigma (P')) + \log (6) + 1.07 + \log (n^{2}\deg (\sigma ))\\
&\le &128 h(E') + \widehat{h}(\sigma (P')) + 67 + 5\log (\deg (\sigma )) + 2\log (n).\\
\end{array}$$
if every prime factor of $B_{n\sigma (P')}$ divides $B_{nP'}$. We conclude the proof applying Lemma \ref{solution-bound-nsquare-logn} with $A=\frac{1}{\sqrt{\widehat{h}(P')}}$.
\end{proof}

\section{Elliptic divisibility sequences and linear forms in elliptic logarithms.}\label{section-quatre}

Since no effective version of Siegel's theorem is known, we can not hope to get an explicit uniform bound on the index of prime power terms in an elliptic divisibility sequence. However an explicit nonuniform bound can be computed using work of David on lower bounds on linear forms in elliptic logarithms.

\begin{snotation}\label{notation-minimal-equation-short-Weierstrass-equation}
We use notation \ref{notation-minimal-equation-short-Weierstrass-equation-1}. We consider the map $\phi$ defined on the unbounded component $\mathcal{E}(\R )_{0}$ of $\mathcal{E}$ by the formula
$$\phi (P)=\phi_{\mathcal{E}}(P):=\textrm{Sign}(\widetilde{y}(P))\int_{\widetilde{x}(P)}^{+\infty}\frac{dt}{\sqrt{t^{3}+at+b}}.$$
The map $\phi$ is linked to the archimedean height by the following inequality (see \cite[section 3, Inequality 2]{Stroeker-Tzanakis}): for every point $P\in\mathcal{E}(\R)_{0}$ we have
\begin{equation}\label{link-archim-height-log-elliptic}
-\log\left|\phi (P)\right|- \frac{1}{2}\log (2) \le h_{\infty}(P)\le -\log\left|\phi (P)\right|+ \frac{5}{2}\log (2).
\end{equation}
Let $\wp$ be the Weierstrass $\wp$-function relative to the elliptic curve $\mathcal{E}$. Let $T_{0}\in\mathcal{E}(\R )$ be the real $2$-torsion point with the highest $x$-coordinate. Let $P\in \mathcal{E}(\Q )$ be a point in the unbounded connected component $\mathcal{E}(\R )_{0}$ of $\mathcal{E}(\R )$. Then $\wp\left(\frac{\phi (P)}{2\phi (T_{0})}\right) = \frac{x(P)}{4}$, and for every $n\in\Z$ there is an integer $m$ such that 
$$\phi (nP)=n\phi (P) + 2m\phi (T_{0}).$$
Moreover, since $|\phi (nP)| < |\phi (T_{0})|$ and $|\phi (P)| < |\phi (T_{0})|$, we have $|m|\le |n|.$
\end{snotation}

\subsection{David's lower bounds on linear forms in elliptic logarithms.}
	
\begin{slemma}\label{Application-Sinnou-David}
Let $E$ be an elliptic curve defined over $\Q$ by a minimal Weierstrass equation with integral coefficients. Let $P\in E(\Q )$ be a point on $E$. For any integer $n>0$ denote by $b_{n}$ the maximum 
$$b_{n}:=\max\left\{\log |2n|,2\widehat{h}(P),12eh(E) +5e\log (6)\right\} .$$
Then for any integer $n>1$ the inequality
$$\widehat{h_{\infty}}(nP)\le c_{1} (b_{n}+\log (3) + 1)^{6} + c_{2}.$$
holds with $c_{1}=5.9\times 10^{43}$ and $c_{2}=h(E) +2.81.$
\end{slemma}
\begin{proof}
We use notation \ref{notation-minimal-equation-short-Weierstrass-equation}. Applying \cite[Th\'eor\`eme 2.1]{David} to the curve $\mathcal{E}$ with $k =2$ and $D\le 3$ and $E=e$ and $\gamma_{1} = P$ and $\gamma_{2} = T_{0}$ and 
$$\begin{array}{rcl}
\log (V_{1})=\log (V_{2})&=& \max\left\{ 2\widehat{h}(P),12 e h(E) + 5e\log (6)\right\}\\
&\ge &\max\left\{ 2\widehat{h}(P), e\max\left\{ 1,h\left( 1,-\frac{a}{4},-\frac{b}{16}\right) ,h(j(\mathcal{E}))\right\},2\pi\sqrt{3}\right\} \\
&\ge &\max\left\{ 2\widehat{h}(P), \max\left\{ 1,h\left( 1,-\frac{a}{4},-\frac{b}{16}\right) ,h(j(\mathcal{E}))\right\} ,\frac{3\pi |\phi (P)|^{2}}{|2\phi (T_{0})|^{2}\textrm{Im}(\tau )}\right\}\\
\end{array}$$
(where $\tau$ is a complex number such that $\mathcal{E}(\C )\simeq \C /(\Z +\tau\Z )$ and $\textrm{Im}(\tau )\ge\frac{\sqrt{3}}{2}$) and 
$$\begin{array}{rcl}
\log (B)&=&\max\{\log |2n|, \log (V_{1})\} \\
&\ge &\max\left\{ e\max\left\{ 1,h\left( 1,-\frac{a}{4},-\frac{b}{16}\right) ,h(j(\mathcal{E}))\right\}, h(n,2m),\frac{\log(V_{1})}{D}\right\} .\\
\end{array}$$
(note that $|m|\le |n|$) we get an inequality
$$\begin{array}{rl}
\log |n\phi (P) +2m\phi (T_{0})|\ge & -C \log (V_{1}) \log (V_{2})(\log (B) + \log (3) + 1)\times\\
&\left(\log (\log (B)) + 12 h(E) +5\log (6) + \log (3) + 1\right)^{3}\\
\end{array}$$
where $C=2.3\times 10^{43}.$ Note that 
\begin{itemize}
\item we do not use the same definition for $\widehat{h}$ as in \cite{David};
\item the number $h:=\max\left\{ 1,h\left( 1,-\frac{a}{4},-\frac{b}{16}\right) ,h(j(\mathcal{E}))\right\}$ is equal the number denoted by $h$ in \cite{David}; Inequality (\ref{bound-David-height-naive-height}) gives an upper bound on $h$ that is linear in $h(E)$ (see notation \ref{notation-minimal-equation-short-Weierstrass-equation-1}).
\end{itemize}
Using the inequalities $\log (x)\le x-1$ (which holds for every real number $x>0$) and 
$$12h(E) +5\log (6) \le e^{-1}\log (V_{1}),$$
we deduce from Inequality~(\ref{link-archim-height-log-elliptic}) that 
$$h_{\infty}(nP)\le (1+e^{-1})^{3}C\left( 1+ \log (3) \log (B)\right)^{6} + \frac{5}{2}\log (2).$$
We conclude by using \cite[Theorem 5.5]{Silverman-diff-height-can-height}.
\end{proof}

\subsection{A nonuniform bound on the index of prime power terms in elliptic divisibility sequences.}

\begin{sproposition}\label{borne-non-uniforme}
We use notation \ref{notation-P-Pprime-sigma}. Then $B_{nP'}$ has a prime factors coprime to $B_{P'}$ for every index 
$$n>\max\left\{2.1\times 10^{30}, \frac{4.3\times 10^{27}}{\widehat{h}(P')},8.7\times 10^{23}\widehat{h}(P')^{5/2},\frac{2\times 10^{27}h(E')^{7/2}}{\widehat{h}(P')}\right\}$$
and $B_{n\sigma (P')}$ has a prime factor coprime to $B_{nP'}$ for every index 
$$n>\max\left\{ 4.2\times 10^{30}, \frac{4.3\times 10^{27}}{\widehat{h}(P')},1.7\times 10^{24}\widehat{h}(\sigma (P'))^{5/2},\frac{4\times 10^{27}h(E)^{7/2}}{\widehat{h}(\sigma (P'))}\right\} .$$
\end{sproposition}
\begin{proof}
Let $n\in\N$ be such that $B_{nP'}$ has no prime factor coprime to $B_{P'}$. Lemma \ref{Application-Sinnou-David} (applied with $b':=\max\left\{ 2\widehat{h}(P');12eh(E')+5e\log (6)\right\}$) asserts that either $\widehat{h_{\infty}}(nP')\le 5.9\times 10^{43}\times (b'+2.1)^{6} + h(E') +2.81$ or $\log |2n| > b'$. We assume for now that $\log |2n|\le b'$. Applying Theorem \ref{theorem-general-bound-PSD-terms} we get that
$$\begin{array}{rcl}
n&\le &\frac{2}{\widehat{h}(P')}+\sqrt{\frac{5.9\times 10^{43}(b'+2.1)^{6}+2h(E')+\widehat{h}(P')+2.81}{\widehat{h}(P')}}\\
&\le &\frac{2+\sqrt{5.91\times 10^{43}(b'+2.1)^{7}}}{\widehat{h}(P')}\\
&\le &\frac{8.7\times 10^{21}\left(\max\left\{\widehat{h}(P')+ 1.05, 17h(E') + 14\right\}\right)^{7/2}}{\widehat{h}(P')}\\
&\le &\frac{8.7\times 10^{21}\left(\max\left\{2\widehat{h}(P'), 34h(E'), 28\right\}\right)^{7/2}}{\widehat{h}(P')}.\\
\end{array}$$
Now we assume that $\log |n|\ge b' .$ The proof of Theorem \ref{theorem-general-bound-PSD-terms} is still valid when $M$ and $M'$ are replaced with polynomials in $\log (n)$. In particular Lemma \ref{Application-Sinnou-David} implies that 
$$\begin{array}[t]{rcl}
n^{2}\widehat{h}(P')&\le & 5.9\times 10^{43} (\log |6n| + 1)^{6} + 2\log (n) + \widehat{h}(P') + 2h(E') + 2.81\\
&\le &2\max\left\{ 5.9\times 10^{43} (\log |6n| +1)^{6}, 2\log (n) + \widehat{h}(P') + 2h(E') + 2.81\right\} .\\
\end{array}$$
Applying Lemma \ref{solution-bound-nsquare-logn} 
\begin{itemize}
\item with $A=10^{18}$ and $d=6$ when $n^{2}\widehat{h}(P')\le 11,8\times 10^{43} (\log |6n| +1)^{6}$,
\item with $A=4d=4$ when $n^{2}\widehat{h}(P') \le 4\log (n) + 2\widehat{h}(P') + 4h(E') + 5.62,$
\end{itemize}
we get that either $n\le \max\left\{ 2.06\times 10^{30}, \frac{4.3\times 10^{27}}{\widehat{h}(P')}\right\}$
$$\begin{array}{rrcl}
\textrm{or}&n&\le &\max\left\{ 16.7, \frac{1}{\widehat{h}(P')}+ \sqrt{ 2 + \frac{5.62}{\widehat{h}(P')}+\frac{4h(E')}{\widehat{h}(P')}}\right\}.\\
&&\le&\max\left\{ 16.7, \left( 1+\sqrt{3}\right)\max\left\{\frac{1}{\widehat{h}(P')}, \sqrt{2},\sqrt{\frac{5.62}{\widehat{h}(P')}}, 2\sqrt{\frac{h(E')}{\widehat{h}(P')}}\right\}\right\}\\
\end{array}$$
In the same way we prove that 
$$\begin{array}{rrcl}
\textrm{either}&n&\le &\frac{1.8\times 10^{22}\left(\max\left\{2\widehat{h}(\sigma (P')), 34h(E), 28\right\}\right)^{5/2}}{\widehat{h}(\sigma (P'))}\\
\textrm{or}&n &\le &\max\left\{ 4.2\times 10^{30}, \frac{8.6\times 10^{27}}{\widehat{h}(\sigma (P'))}\right\}\\
\textrm{or}&n&\le &\max\left\{ 16.7, \frac{1}{(\deg (\sigma )-1)\widehat{h}(P')}+ \sqrt{ 4 + \frac{5.62+2\log (\deg (\sigma ))}{(\deg (\sigma )-1)\widehat{h}(P')}+\frac{8h(E)}{\widehat{h}(\sigma (P'))}}\right\}\\
&&\le &\max\left\{ 16.7, \left( 1+\sqrt{3}\right)\max\left\{\frac{2}{\widehat{h}(\sigma (P'))},\sqrt{\frac{7.62}{\widehat{h}(P')}}, \sqrt{\frac{8h(E)}{\widehat{h}(\sigma (P'))}}\right\}\right\}\\
\end{array}$$
whenever $B_{n\sigma (P')}$ has no prime factor coprime to $B_{nP'}$. 
\end{proof}

\subsection{An explicit version of the gap principle.}

David's theorem about lower bounds on linear forms in elliptic logarithm leads to a bound $M(B)$ on the index of prime terms in a magnified elliptic divisibility sequence $B$ that is quite large. As explained for example in \cite{Tzanakis-De-Weger}, the bound $M(B)$ can be reduced applying the LLL algorithm or Mumford's gap principle. 

\begin{snotation}\label{notation-gap-principle}
We use notation \ref{notation-P-Pprime-sigma}. Following notation \ref{notation-minimal-equation-short-Weierstrass-equation} we denote by $\mathcal{E}$ (respectively $\mathcal{E'}$) a model of $E$ (respectively $E'$) given by a short Weierstrass equation with coefficients in $\Z$ such that $\Delta_{\mathcal{E}}=6^{12}\Delta_{E}$ (respectively $\Delta_{\mathcal{E'}}~=~6^{12}\Delta_{E'}$). Let $P'$ be a $\Q$-point on $E'$. We denote by $R'\in\mathcal{E}'(\Q)$ (respectively $R~\in~\mathcal{E}(\Q)$)  the point on $\mathcal{E}'$ (respectively $\mathcal{E}$) associated to $P'$ (respectively $\sigma (P' )$).
\end{snotation}

\begin{slemma}\label{explicit-LLL-phin-nphi} We use notation \ref{notation-gap-principle}. Let
\begin{equation}\label{hypothese-n-Lemme-explicit-LLL-phin-nphi}
n >\max\left\{ 8,\frac{2}{\widehat{h}(P')}+\sqrt{3+\max\left\{\frac{5h(E')}{\widehat{h}(P')}, \frac{9h(E)}{\widehat{h}(P)}\right\} + \frac{7}{\widehat{h}(P')} }\right\} .
\end{equation}
be such that $B_{n\sigma (P')}$ has at most one prime factor coprime to $B_{P'}$.
\begin{enumerate}
\item Assume every prime factor of $B_{nP'}$ divides $B_{P'}$. Then we have 
$|x(nR')|~\ge~2\max\left\{ |x(T)| : T\in\mathcal{E}'[2]\right\}$ and $n\phi_{\mathcal{E}'}(R')\neq \phi_{\mathcal{E}'}(nR')$;
\item[ ]
\item Assume every prime factor of $B_{n\sigma (P')}$ divides $B_{nP'}$. Then we have 
$|x(nR)|~\ge~2\max\left\{ |x(T)| : T\in\mathcal{E}'[2]\right\}$ and $n\phi_{\mathcal{E}}(R)\neq \phi_{\mathcal{E}}(nR)$.
\end{enumerate}
\end{slemma}
\begin{proof} 
When $|x (nR')|\le 2\max\left\{ |x(T)| : T\in\mathcal{E}'[2]\right\} ,$ Lemma \ref{weakerthan-canonical-height} gives
$$\widehat{h_{\infty}}(nP')= \widehat{h_{\infty}}(nR')\le 3h(E') + \log (6) + \frac{1}{2}\log (2) + 1.07\le 3h(E') + 3.21$$
Thus Theorem \ref{theorem-general-bound-PSD-terms} implies that,  if every prime factor of $B_{nP'}$ divides $B_{P'}$ and $n> \frac{2}{\widehat{h}(P')} + \sqrt{1 +\frac{4 h(E') +3.21}{\widehat{h}(P')} }$, then $|x (nR')|~\ge~2\max\left\{ |x(T)| : T\in\mathcal{E}'[2]\right\}$ (which implies that $R'\in\mathcal{E}'(\R )_{0}$).

In the same way, since $n> \frac{2}{\widehat{h}(P')} + \sqrt{2 +\frac{8 h(E)}{\widehat{h}(P)} + \frac{4.21}{\widehat{h}(P')} }$, we deduce from Theorem \ref{theorem-general-bound-PSD-terms} that, if every prime factor of $B_{n\sigma (P')}$ divides $B_{nP'}$, then we have $|x (nR)|\ge 2\max\left\{ |x(T)| : T\in\mathcal{E}[2]\right\}$ (and in particular $R\in\mathcal{E}(\R )_{0}$).

Assume that $|x (nR')|\ge 2\max\left\{ |x(T)| : T\in\mathcal{E'}[2]\right\}$ and $n\phi_{\mathcal{E}'}(R')\neq \phi_{\mathcal{E}'}(nR')$ Then Inequality (\ref{link-archim-height-log-elliptic}) gives 
$$\begin{array}{rcl}
h_{\infty}(nR') - \frac{5}{2}\log (2)&\le &- \log\left|\phi_{\mathcal{E}'}(nR')\right|\\
&\le &-\log (n) -\log\left|\phi_{\mathcal{E}'}(R')\right|\\
&\le &-\log (n) + h_{\infty}(R') + \frac{1}{2}\log\left( 2\right) .\\
\end{array}$$
Now \cite[Theorem 1.1]{Silverman-diff-height-can-height} asserts that 
$$\begin{array}{rcl}
h_{\infty}(R')\le h(R')&\le &\widehat{h}(R') + h(\mathcal{E}') + \frac{3}{24}h(j(\mathcal{E}')) + 0.973\\
&\le &\widehat{h}(R') + \frac{5}{2}h(E') +\log (6) + 0.973.\\
\end{array}$$
Applying \cite[Theorem 5.5]{Silverman-diff-height-can-height} to $\widehat{h_{\infty}}(nP') = \widehat{h_{\infty}}(nR')$ we get 
\begin{equation}\label{explicit-LLL-phin-nphi-last-equation}
\widehat{h_{\infty}}(nP') + \log (n)\le\widehat{h}(P') + \frac{7}{2}h(E') + 2\log (6) + 3\log (2) + 2.043.
\end{equation}
If every prime factor of $B_{nP'}$ divides $B_{P'}$ and $n\phi_{\mathcal{E}'}(R')\neq \phi_{\mathcal{E}'}(nR')$ and $3\log (2)\le \log (n)$ then it follows from Inequation (\ref{explicit-LLL-phin-nphi-last-equation}) and Theorem \ref{theorem-general-bound-PSD-terms} that $n\le \frac{2}{\widehat{h}(P')} + \sqrt{2 + \frac{5h(E') + 6}{\widehat{h}(P')}}.$ The proof for Inequality  (\ref{explicit-LLL-phin-nphi-last-equation}) holds also when replacing $E'$, $P'$ and $R'$ respectively by $E,$ $P$ and $R$. It follows that if every prime factor of $B_{n\sigma (P')}$ divides $B_{nP'}$ and $n\phi_{\mathcal{E}}(R)\neq \phi_{\mathcal{E}}(nR)$ and $3\log (2)\le \log (n)$ then $n\le \frac{2}{\widehat{h}(P')} + \sqrt{3 + \frac{9h(E')}{\widehat{h}(P)} + \frac{7}{\widehat{h}(P')}}.$
\end{proof}

\begin{sproposition}\label{proposition-finle-gap-principle}
We use notation \ref{notation-P-Pprime-sigma} and  we assume that $E$ and $E'$ are given by minimal Weierstrass equations. Let $n_{3}>n_{2}>n_{1}>8$ be three pairwise coprime integers with
\begin{equation}\label{condition-borne-min-gap}
n_{3}>n_{2}>n_{1}>\frac{2}{\widehat{h}(P')}+\sqrt{3+\max\left\{\frac{5h(E')}{\widehat{h}(P')}, \frac{9h(E)}{\widehat{h}(P)}\right\} + \frac{7}{\widehat{h}(P')} }
\end{equation}
such that $B_{n_{i}P}$ has at most one prime factor coprime to $B_{P'}$. Then we have 
$$\begin{array}{cc}
\textrm{either }&n_{1}\le\frac{2}{\widehat{h}(P')} + \sqrt{2+ \frac{2\log (n_{3}) + 52h(E)}{\widehat{h}(\sigma (P'))} +\frac{24.42}{\widehat{h}(P')}}\\
\textrm{or}&n_{1}\le\frac{2}{\widehat{h}(P')} + \sqrt{1+\frac{\log (n_{i}) + 26h(E') + 23.42}{\widehat{h}(P')}}\\
\end{array}$$
with $i\in\{2,3\}$ an index such that every prime factor of $B_{n_{i}P'}$ divides $B_{P'}.$
\end{sproposition}
\begin{proof} We use notation \ref{notation-gap-principle}. For every $l\in\{ 1,2,3\}$ at most one prime factor of $B_{n_{l}\sigma (P')}$ does not divide $B_{P'}$. In particular two indices $i\neq j$ are such that
\begin{itemize}
\item either every prime factor of $B_{n_{i}P'}$ divides $B_{P'}$ and every prime factor of $B_{n_{j}P'}$ divides $B_{P'}$; 
\item or every prime factor of $B_{n_{i}\sigma (P')}$ divides $B_{n_{i}P'}$ and every prime factor of $B_{n_{j}\sigma (P')}$ divides $B_{n_{j}P'}$. 
\end{itemize}
We assume for now that every prime factor of $B_{n_{i}P'}$ divides $B_{P'}$ and every prime factor of $B_{n_{j}P'}$ divides $B_{P'}$. Lemma \ref{explicit-LLL-phin-nphi} asserts that
\begin{itemize}
\item $\left|x(n_{i}P')\right| \ge 2\max\left\{ |x(T)|~:~T\in\mathcal{E}'[2]\right\}$ and $\phi_{\mathcal{E}'}(n_{i}P') \neq n_{i}\phi_{\mathcal{E}'}(P')$;
\item $\left|x(n_{j}P')\right| \ge 2\max\left\{ |x(T)|~:~T\in\mathcal{E}'[2]\right\}$ and $\phi_{\mathcal{E}'}(n_{j}P') \neq n_{j}\phi_{\mathcal{E}'}(P')$.
\end{itemize}
We denote by $m_{i}\neq 0$ and $m_{j}\neq 0$ two integers such that 
$$\begin{array}{cc}
&\phi_{\mathcal{E}'} (n_{i}P')=n_{i}\phi_{\mathcal{E}'} (P') + 2m_{i}\phi_{\mathcal{E}'} (T_{0})\\
\textrm{and}&\phi_{\mathcal{E}'} (n_{j}P')=n_{j}\phi_{\mathcal{E}'} (P') + 2m_{j}\phi_{\mathcal{E}'} (T_{0}).\\
\end{array}$$
Since $\left| n_{i}\phi_{\mathcal{E}'} (P') + 2m_{i}\phi_{\mathcal{E}'} (T_{0})\right| \le \left|\phi_{\mathcal{E}'} (T_{0})\right|$ and $|\phi_{\mathcal{E}'} (P')|\le\left|\phi_{\mathcal{E}'} (T_{0})\right|$ we have $|m_{i}|<|n_{i}|$. However if $n_{i}m_{j}=n_{j}m_{i}$ then $n_{i}$ is a divisor of $m_{i}$ (because $n_{i}$ and $n_{j}$ are coprime). It follows that $n_{j}m_{i}-n_{i}m_{j}\neq 0.$ In particular we get 
\begin{equation}\label{inegalite-fondamentale-gap-principle}
\begin{array}{rcl}
2\left|\phi_{\mathcal{E}'} (T_{0})\right| &\le &2\left|\phi_{\mathcal{E}'} (T_{0})\right|\left| n_{j}m_{i}-n_{i}m_{j}\right|\\
&\le &\left|n_{j}\phi_{\mathcal{E}'} (n_{i}P') - n_{i}\phi_{\mathcal{E}'} (n_{j}P') \right|\\
&\le &2\max\left\{ |n_{j}|\left|\phi_{\mathcal{E}'} (n_{i}P')\right| ,|n_{i}|\left|\phi_{\mathcal{E}'} (n_{j}P')\right|\right\}\\
\end{array}
\end{equation}
We deduce from Inequality (\ref{link-archim-height-log-elliptic}) and Inequality (\ref{inegalite-fondamentale-gap-principle}) that 
$$\min\left\{ h_{\infty}(n_{j}P') - \log (n_{i}),h_{\infty}(n_{i}P')-\log (n_{j})\right\}\le  - \log\left|\phi_{\mathcal{E}'}\left( T_{0}\right)\right| + \frac{5}{2}\log (2).$$
Applying \cite[Lemme 2.1]{Pellarin} (and Inequality (\ref{bound-David-height-naive-height})) we get
\begin{equation}\label{inegalite-fondamentale-gap-principle-2}
\min\left\{ h_{\infty}(n_{j}P') - \log (n_{i}), h_{\infty}(n_{i}P')-\log (n_{j})\right\}\le 24h(E') + 22.35.
\end{equation}
Theorem \ref{theorem-general-bound-PSD-terms} and \cite[Theorem 5.5]{Silverman-diff-height-can-height}  show that
$$n_{1}\le\min\{ n_{i},n_{j}\}\le\frac{2}{\widehat{h}(P')} + \sqrt{1+\frac{\log\left(\max\left\{ n_{i},n_{j}\right\}\right) + 26h(E') + 23.42}{\widehat{h}(P')}}$$
Now we assume that every prime factor of $B_{n_{i}\sigma (P')}$ divides $B_{n_{i}P'}$ and every prime factor of $B_{n_{j}\sigma (P')}$ divides $B_{n_{j}P'}.$ An analog argument show that
$$\min\left\{ h_{\infty}(n_{j}\sigma (P')) - \log (n_{i}), h_{\infty}(n_{i}\sigma (P'))-\log (n_{j})\right\}\le 24h(E) + 22.35.$$ 
From this inequality and Theorem \ref{theorem-general-bound-PSD-terms} and \cite[Theorem 5.5]{Silverman-diff-height-can-height}  we deduce that
$$n_{1}\le\frac{2}{\widehat{h}(P')} + \sqrt{2+\frac{2\log (n_{3}) +  52h(E)}{\widehat{h}(\sigma (P'))}+\frac{23.42}{\widehat{h}(P')} + \frac{\log (d)}{(d-1)\widehat{h}(P')}}$$
(note that $n_{1}\le\min\{ n_{i},n_{j}\}$ and $\max\{ n_{i},n_{j}\}\le n_{3}$).
\end{proof}

\subsection{The proof of Theorem \ref{Bound-compositecase-N1-N3}}

The inequality $h(E')\ge \frac{1}{12}\log (2)$ implies that $\frac{2}{\widehat{h}(\sigma (P'))}\le \frac{1}{\widehat{h}(P')}\le\frac{12 C}{\log (2)}\le 17.32\times C$ with $C:=\max\left\{1, \frac{h(E')}{\widehat{h}(P')},\frac{h(E)}{\widehat{h}(\sigma (P'))}\right\}.$

Let $n$ be an integer such that  at most one prime factor $B_{n\sigma (P')}$ is not a prime factor of $B_{P'}$. If $n = n_{1}n_{2}$ with $n_{1} \ge n_{2}>1$, then Corollary \ref{application-cas-doublement-magnifie} implies that either $n\le n_{1}^{2}\le \frac{4}{\widehat{h}(P')}\left(\frac{1}{2}\log\left(\frac{4}{\widehat{h}(P')}\right)\right)^{2}\le 18C\left(\log (70 C)\right)^{2}$ or 
$$n\le n_{1}^{2}\le\left(\frac{144}{\sqrt{\widehat{h}(P')}}+2\sqrt{1+\frac{128h(E')+135}{\widehat{h}(P')}}\right)^{2}\le 490000C$$
Proposition \ref{borne-non-uniforme} asserts that 
$$N_{1}\le \max\left\{ 4.2\times 10^{30}C, 1.7\times 10^{24}\widehat{h}(\sigma (P'))^{5/2}, 4\times 10^{27}C^{7/2}\widehat{h}(\sigma (P'))^{5/2}\right\} .$$
In particular (since $h\ge \log (h)$ for every $h\ge 1$) we have 
\begin{equation}\label{borne-maximal-prime-term}
\frac{\log (N_{1})}{\widehat{h}(\sigma (P'))}\le 600 C + 31C\log (C) +\frac{5}{2}
\end{equation}
Noticing that $\frac{2}{\widehat{h}(P')}+\sqrt{3+\max\left\{\frac{5h(E')}{\widehat{h}(P')}, \frac{9h(E)}{\widehat{h}(P)}\right\} + \frac{7}{\widehat{h}(P')}}\le 47C,$ we deduce from Proposition \ref{proposition-finle-gap-principle} and Inequality (\ref{borne-maximal-prime-term}) that either $N_{3}\le 47C$
\begin{eqnarray}
\textrm{or}&N_{3}\le\frac{2}{\widehat{h}(P')} + \sqrt{2+ \frac{2\log (N_{1}) + 52h(E)}{\widehat{h}(\sigma (P'))} +\frac{24.42}{\widehat{h}(P')}}\le 77 C\\
\textrm{or}&N_{3}\le\frac{2}{\widehat{h}(P')} + \sqrt{1+\frac{\log (N_{i}) + 26h(E') + 23.42}{\widehat{h}(P')}}\label{borne-finale-N3-ter}
\end{eqnarray}
where $i\in \{ 1,2\}$ is such that every prime factor of $B_{N_{i}P'}$ divides $B_{P'}.$ When Inequality (\ref{borne-finale-N3-ter}) holds Proposition \ref{borne-non-uniforme} gives $\frac{\log (N_{i})}{\widehat{h}(P')}\le 1202 C + 62C\log (C).$
In that case  Inequality (\ref{borne-finale-N3-ter}) implies that $N_{3}\le 77 C.$ We conclude applying the main result in \cite{hindrysilverman}.

\section{Elliptic curves with $j$-invariant 1728.}\label{section-six}

In this section we compte the bound from Corollary~\ref{cas-composante-bornee} in the particular case of an elliptic curve $E_{A}$ defined by a Weierstrass equation
$$E_{A}:y^{2}=x(x^{2}-A)$$
where $A$ denotes a positive integer with no valuation greater than or equal to 4. For congruent number curves such values can be deduced easily from results on integer points on $E_{N^{2}}$. In the case $A\notin\Q^{\times 2}$, the main issue is to get an explicit version of Lang's conjecture (which is known to be true for elliptic curves with an integral $j$-invariant). 

\begin{proposition}\label{proposition-controle-hauteur-canonique-deg2}
Let $P\in E_{A}(\Q )$ be a nontorsion point lying on the unbounded connected component of $E_{A}(\R )$. Denote by $\widehat{h}_{A}$ the canonical height for $E_{A}$. Then 
\begin{equation}\label{explicit-Lang-P-Anot12mod16-deg2}
\widehat{h}_{A}(P)\ge\frac{1}{16}\log |2A|.
\end{equation}
when $A\not\equiv 12\bmod 16$ and 
\begin{equation}\label{explicit-Lang-P-A12mod16-deg2}
\widehat{h}_{A}(P)\ge\frac{1}{64}\log |2A|.
\end{equation}
when $A\equiv 12\bmod 16.$ Moreover we have:
\begin{equation}\label{explicit-difference-hauteur-hauteur-canonique-deg2}
-\frac{1}{4}\log |A|-\frac{3}{8}\log (2)\le\widehat{h}_{A}(P) -\frac{1}{4}\log |A_{P}^{2}+AB_{P}^{4}|\le\frac{1}{12}\log (2)
\end{equation}
\end{proposition}
\begin{proof}The proposition is similar to \cite[Proposition 2.1]{bremsil} so we do not give a full proof here. However more reduction types have to be considered leading to a more complicated proof. The proof is based on the decomposition of the canonical height into a sum of local canonical heights. 

Denote by $\Delta_{A}=64A^{3}$ the discriminant of $E_{A}$. The contribution of the archimedean height is computed using Tate's series as in \cite{bremsil}. We get 
\begin{equation}\label{calcul-deg2-hauteur-archim}
0\le\widehat{h}_{\infty}(P)-\frac{1}{4}\log |x(P)^{2}+ A| + \frac{1}{12}\log (\Delta_{A})\le \frac{1}{12}\log (2).
\end{equation}
Non-archimedean canonical heights are computed using the algorithm presented in \cite{Silv-compute-can-height}. If $v$ is an odd prime number, then Tate's algorithm can be used to prove that $E_{A}$ has reduction type:
\begin{itemize}
\item $I_{0}$ at $v$ when $\ord_{v}(A)=0$;
\item $III$ at $v$ when $\ord_{v}(A)=1$;
\item $I_{0}^{*}$ at $v$ when $\ord_{v}(A)=2$;
\item $III^{*}$ at $v$ when $\ord_{v}(A)=3$;
\end{itemize}
In particular $2P$ has always good reduction at $v$ and we get 
\begin{equation}\label{calcul-deg2-hauteur-p-adic}
\begin{array}{rcl}
-\frac{v(A)}{4}\le \widehat{h_{v}}(P)-\frac{1}{2}\max\{ 0,-v(x(P))\} -\frac{v(\Delta_{A})}{12}\le 0\\
\end{array}
\end{equation}
(the only technical issue is the case $\ord_{v}(A) = 2\ord_{v}(x(P))=2$; in that case the equation for $E_{A}$ implies that $\ord_{v}(x(P)^{2}-A)\equiv\ord_{v}(x(P))\bmod 2$ and it follows that $\ord_{v}(x(P)^{2}+A)= \ord_{v}(2A)=2$).

Considering the specialization of $E_{A}$ at $2$, Tate's Algorithm gives a reduction type:
\begin{itemize}
\item $II$ for $E_{A}$ at $2$ when $A\equiv -1\bmod 4$;
\item $III$ for $E_{A}$ at $2$ when $A\equiv 1\bmod 4$;
\item $III$ for $E_{A}$ at $2$ when $\ord_{2}(A)=1$;
\item $I_{2}^{*}$ for $E_{A}$ at $2$ when $A\equiv 4\bmod 16$;
\item $I_{3}^{*}$ for $E_{A}$ at $2$ when $A\equiv 12\bmod 16$;
\item $III^{*}$ for $E_{A}$ at $2$ when $\ord_{2}(A)=3$;
\end{itemize}
In particular every double in $E_{A}(\Q )$ has good reduction everywhere if and only if $A\not\equiv~12\bmod 16$. When $A\equiv 12\bmod 16$, every $\Q$-point on $E_{A}$ in the image of the multiplication-by-4 map has good reduction everywhere. Moreover the algorithm described in \cite{Silv-compute-can-height} gives 
\begin{equation}\label{calcul-deg2-hauteur-2-adic}
-\frac{v_{2}(A)}{4}-\frac{3}{8}\log (2)\le\widehat{h_{2}}(P)-\frac{1}{2}\max\{ 0,-v_{2}(x(P))\} -\frac{v_{2}(\Delta_{A})}{12}\le 0
\end{equation}
We compute the canonical height by summing local canonical heights. Doing so Inequality (\ref{explicit-difference-hauteur-hauteur-canonique-deg2}) becomes a consequence of inequations (\ref{calcul-deg2-hauteur-archim}), (\ref{calcul-deg2-hauteur-p-adic}) and~(\ref{calcul-deg2-hauteur-2-adic}).\\

Now we prove the two inequalities (\ref{explicit-Lang-P-Anot12mod16-deg2}) and (\ref{explicit-Lang-P-A12mod16-deg2}). When  $Q\in E_{A}(\Q )$ has good reduction everywherewe we have
$$\displaystyle\sum_{v\ne\infty}\widehat{h_{v}}(Q) = \log |B_{Q}| +\frac{1}{4}\log |4A|.$$
Adding this equation and the inequation (\ref{calcul-deg2-hauteur-archim}) we get 
\begin{equation}\label{Lang-EA-firstbound-proof}
\widehat{h}_{A}(Q)\ge\frac{1}{4}\log |A_{Q}^{2}+AB_{Q}^{4}|
\end{equation}
If $Q$ is a point in the unbounded real connected component of $E_{A}$, then $|A_{Q}| = |x(Q)| B_{Q}^{2}\ge\sqrt{|A|}B_{Q}^{2}\ge\sqrt{|A|}$. Inequality (\ref{Lang-EA-firstbound-proof}) becomes
\begin{equation}\label{explicit-Lang-Q-deg2}
\widehat{h}_{A}(Q)\ge\frac{1}{4}\log |2A|.
\end{equation}
Let $P$ be any $\Q$-point on $E_{A}$. As shown above $2P$ has good reduction everywhere whenever $A\not\equiv 12\bmod 16$ and $4P$ has good reduction everywhere. The two inequalities (\ref{explicit-Lang-P-Anot12mod16-deg2}) and (\ref{explicit-Lang-P-A12mod16-deg2}) follow from Inequation (\ref{explicit-Lang-Q-deg2}) applied with $Q\in\{ 2P,4P\}$.
\end{proof}

\begin{proposition}\label{PC-cas-de-[2]-en-degre-2}
Let $P$ be a $\Q$-point on $E_{A}$. Then $B_{2nP}$ is composite in the two following cases:
\begin{itemize}
\item when $n\ge 10$ and $A\not\equiv 12\bmod 16$;
\item when $k\ge 5$ and $A\equiv 12\bmod 16$;
\end{itemize}
\end{proposition}
\begin{proof}To simplify the proof we assume (without loss of generality) that $B_{nP}>0$. Since $\gcd (A_{kP},B_{kP})=1$ the equality 
$$x(2kP)=\frac{(A_{kP}^{2}+AB_{kP}^{4})^{2}}{4B_{kP}^{2}A_{kP}(A_{kP}^{2}-AB_{kP}^{4})}$$
shows that $B_{2kP}$ is composite in the three following cases:
\begin{itemize}
\item when $B_{kP}>1$ and $|A_{kP}|>A^{2}$;
\item when $B_{kP}>1$ and $AB_{kP}^{4}-A_{kP}^{2}>4A^{2}$;
\item 
when $|A_{kP}|>A^{3}$ and $A_{kP}^{2}-AB_{k}^{4}>4A^{2}$.
\end{itemize}
(note that $4A^{2}\ge\gcd\left( AB_{kP}^{4}-A_{kP}^{2}, (A_{kP}^{2}+AB_{kP}^{4})^{2}\right)$).
We assume without loss of generality that we are not in the first case i.e. that either $B_{kP}=1$ or $|A_{kP}|\le A^{2}$. We show the second case happens whenever $x(kP)<0$ and the third case happens whenever $x(kP)>0$.\\

\noindent\textbf{Case 1: if $x(kP)<0$.} Then $|x(kP)|<\sqrt{|A|}$ which implies that 
$$\log |A_{kP}^{2}+AB_{kP}^{4}|\le \log (2AB_{kP}^{4}).$$
Now Inequation (\ref{explicit-difference-hauteur-hauteur-canonique-deg2}) gives:
$$k^{2}\widehat{h}_{A}(P)\le \frac{1}{4}\log (2AB_{kP}^{4}) + \frac{1}{12}\log (2).$$
Using Inequations (\ref{explicit-Lang-P-Anot12mod16-deg2}) and (\ref{explicit-Lang-P-A12mod16-deg2}) we get 
\begin{equation}
\frac{k^{2}}{16}\log(2A)\le \frac{1}{4}\log (2AB_{kP}^{4}) + \frac{1}{12}\log (2)
\end{equation}
when $A\not\equiv 12\bmod 16$ and 
\begin{equation}
\frac{k^{2}}{64}\log(2A)\le \frac{1}{4}\log (2AB_{kP}^{4}) + \frac{1}{12}\log (2)
\end{equation}
when $A\equiv 12\bmod 16$. In particular: 
\begin{itemize}
\item the inequality $B_{kP}>1$ holds for $k\ge 3$ when $A\not\equiv 12\bmod 16$, and for $k\ge 5$ when $A\equiv 12\bmod 16$;\\
\item the inequality $AB_{kP}^{4}>5A^{2}$ holds for $k\ge 4$ when $A\not\equiv 12\bmod 16$, and for $k\ge 8$ when $A\equiv 12\bmod 16$.
\end{itemize}
Note that if $|B_{kP}|>1$ then (by assumption) $|A_{kP}|\le A^{2}$. It follows that the inequality 
$$AB_{k}^{4}-A_{kP}^{2}\ge AB_{kP}^{4}-A^{4}>4A^{2}$$ 
holds when $|B_{kP}|>1$ and $AB_{kP}^{4}>5A^{4}\ge4A^{2}+A^{4}$. \\

\noindent\textbf{Case 2: if $x(kP)>0$.} Then $|x(kP)|>\sqrt{|A|}$ which implies that 
$$\log |A_{kP}^{2}+AB_{kP}^{4}|\le 2\log |2A_{kP}| - \log (2).$$
Now Inequation (\ref{explicit-difference-hauteur-hauteur-canonique-deg2}) gives:
$$k^{2}\widehat{h}_{A}(P)\le \frac{1}{2}\log |2A_{kP}| - \frac{1}{6}\log (2).$$
Using Inequations (\ref{explicit-Lang-P-Anot12mod16-deg2}) and (\ref{explicit-Lang-P-A12mod16-deg2}) we get 
\begin{equation}
\frac{k^{2}}{16}\log(2A)\le \frac{1}{2}\log |2A_{kP}| - \frac{1}{6}\log (2)
\end{equation}
when $A\not\equiv 12\bmod 16$ and 
\begin{equation}
\frac{k^{2}}{64}\log(2A)\le \frac{1}{2}\log |2A_{kP}| - \frac{1}{6}\log (2)
\end{equation}
when $A\equiv 12\bmod 16$. In particular: 
\begin{itemize}
\item the inequality $|A_{kP}|>A^{3}$ holds for  $k\ge 5$ if $A\not\equiv 12\bmod 16$, and for $k\ge 10$ if $A\equiv 12\bmod 16$;\\
\item The inequality $A_{kP}^{2}>5A^{2}$ holds for  $k\ge 3$ if $A\not\equiv 12\bmod 16$, and for $k\ge 6$ if $A\equiv 12\bmod 16$.
\end{itemize}
Suppose $|A_{kP}|~>~A^{3}$. Then $|A_{kP}|~>~A^{2}$ and it follows that $B_{kP}=1$. In particular the inequality 
$$A_{kP}^{2}-AB_{k}^{4}\ge A_{kP}^{2}-A^{2}>4A^{2}$$ 
holds when $A_{kP}^{2}>5A^{2}\ge 4A^{2}+A$ and $|A_{kP}|~>~A^{3}$. 
\end{proof}

\begin{proposition} Let $m$ be an odd integer. Let $P'$ be a $\Q$-point on $E_{A}$. Denote by $P$ the multiple $mP'$. Assume $P\in E_{A}(\Q )$ is a point on the bounded component of $E_{A}$. Then $B_{nP}$ is composite:
\begin{itemize}
\item when $n\ge 4$ and $A\not\equiv 12\bmod 16$;
\item when  $n\ge 8$ and $A\equiv 12\bmod 16$.
\end{itemize}
\end{proposition}
\begin{proof} 
When $n$ is even, Proposition \ref{PC-cas-de-[2]-en-degre-2} applied to $P'$ shows $B_{nP}=B_{nmP'}$ is composite:
\begin{itemize}
\item when $n\ge\frac{10}{m}$ and $A\not\equiv 12\bmod 16$;
\item when $n\ge\frac{20}{m}$ and $A\equiv 12\bmod 16$.
\end{itemize}
From now on we assume that $n$ is odd. In that case $nP$ lies on the
bounded component of the curve. As in the proof of Proposition \ref{PC-cas-de-[2]-en-degre-2} this implies that
\begin{equation}\label{Siegel-effectif-[m]-deg-2-pour-Pprime}
n^{2}\widehat{h}_{A}(P')\le \log (B_{nP'}) + \frac{1}{4}\log (2A) + \frac{1}{12}\log (2)\textrm{ and}
\end{equation}
\begin{equation}\label{Siegel-effectif-[m]-deg-2-pour-P}
m^{2}n^{2}\widehat{h}_{A}(P')\le \log (B_{nP}) + \frac{1}{4}\log (2A) + \frac{1}{12}\log (2).
\end{equation}
Equation (\ref{Siegel-effectif-[m]-deg-2-pour-Pprime}) shows that the inequality $B_{nP'}>1$ holds for $n\ge 3$ when $A\not\equiv 12\bmod 16$ and for $n\ge 6$ when $A\equiv 12\bmod 16$.

From now on we assume that each prime factor of $B_{nP}$ divides $B_{nP'}$. Then \cite[Lemma 2.3]{eims} implies that $B_{nP}$ divides $m^{2}B_{nP'}$. As a consequence Equation (\ref{Siegel-effectif-[m]-deg-2-pour-P}) gives
$$m^{2}n^{2}\widehat{h}_{A}(P')\le 2\log (m) + \frac{1}{4}\log (B_{nP'}^{4}) + \frac{1}{4}\log (A) + \frac{1}{3}\log (2).$$
Using the first inequality in Inequations (\ref{explicit-difference-hauteur-hauteur-canonique-deg2}) we get 
$$\begin{array}{rcl}
m^{2}n^{2}\widehat{h}_{A}(P')&\le &\frac{1}{4}\log |A_{nP'}^{2}+AB_{nP'}^{4}| + 2\log (m) + \frac{1}{3}\log (2)\\
&\le & n^{2}\widehat{h}_{A}(P') + \frac{1}{4}\log |A| + 2\log (m) + \frac{17}{24}\log (2)\\
\end{array}$$
Now it follows from Inequations (\ref{explicit-Lang-P-Anot12mod16-deg2}) and (\ref{explicit-Lang-P-A12mod16-deg2}) that 
$$\frac{(m^{2}-1)n^{2}}{16}\log |2A|\le\frac{1}{4}\log |2A| + 2\log (m) + \frac{11}{24}\log (2)$$
when $A\not\equiv 12\bmod 16$ and 
$$\frac{(m^{2}-1)n^{2}}{64}\log |2A|\le\frac{1}{4}\log |2A| + 2\log (m) + \frac{11}{24}\log (2)$$
when $A\equiv 12\bmod 16$. Since $m\ge 3$ those inequations imply $n<4$ when $A\not\equiv 12\bmod 16$, and $n<8$ when $A\equiv 12\bmod 16$.
\end{proof}

\end{document}